\documentclass[11pt,reno]{amsart} %Change back later
\usepackage{amssymb}
\usepackage{amsmath}
\usepackage{enumitem}
\usepackage{mathrsfs}
\usepackage[english]{babel}
\usepackage[all]{xy}
\usepackage{hyperref}
\usepackage{marginnote}
\usepackage{mathtools}

\usepackage{tikz-cd}
\usepackage{tikz}
\usetikzlibrary{arrows.meta, cd, decorations.markings, calc, intersections}
\usepackage{fullpage}

\usepackage{stmaryrd}

\usepackage[margin=1.1in]{geometry}

\newtheorem{theorem}{Theorem}[section]
\newtheorem{lemma}[theorem]{Lemma}
\newtheorem{proposition}[theorem]{Proposition}
\newtheorem{corollary}[theorem]{Corollary}
\newtheorem{conjecture}[theorem]{Conjecture}
\newtheorem{alphtheorem}{Theorem}

\theoremstyle{definition}
\newtheorem*{ack}{Acknowledgements}
\newtheorem*{con}{Conventions}
\newtheorem{remark}[theorem]{Remark}

\newtheorem{definition}[theorem]{Definition}

\numberwithin{equation}{section} \numberwithin{figure}{section}

\DeclareMathOperator{\Spec}{Spec}

\DeclareMathOperator{\an}{an}

\DeclareMathOperator{\supp}{supp}
\DeclareMathOperator{\red}{red}

\DeclarePairedDelimiter\powerseries{[ \! [}{] \! ]}
\DeclarePairedDelimiter\laurent{( \! (}{) \! )}

\newcommand*\ratmap{\mathbin{\tikz [baseline=0ex,-latex, dashed, ->] \draw [densely dashed] (0em,0.58ex) -- (1.3em,0.58ex);}}

\newcommand{\et}{\textrm{\'{e}t}}

\usepackage{color}

\definecolor{orange}{rgb}{1,0.5,0}

\RequirePackage[normalem]{ulem}

\date{October 11th, 2024; revised on October 16th, 2025}
\title{The Weakly Special Conjecture contradicts orbifold Mordell, and hence the abc conjecture}

\author{Finn Bartsch}
\address{Finn Bartsch \\
IMAPP Radboud University Nijmegen \\
PO Box 9010, 6500GL \\
Nijmegen, The Netherlands\\}
\email{f.bartsch@math.ru.nl }

\author{Fr\'ed\'eric Campana}
\address{Fr\'ed\'eric Campana,
Institut Élie Cartan de Lorraine,
Université de Lorraine, Site de Nancy,
B.P. 70239, F-54506 Vandoeuvre-lès-Nancy Cedex,
France}
\email{frederic.campana@univ-lorraine.fr}

\author{Ariyan Javanpeykar}
\address{Ariyan Javanpeykar,
IMAPP Radboud University Nijmegen,
PO Box 9010, 6500GL,
Nijmegen, The Netherlands}
\email{ariyan.javanpeykar@ru.nl}

\author{Olivier Wittenberg}
\address{Olivier Wittenberg, 
Laboratoire Analyse, Géométrie et Applications,
Institut Galilée - Université Sorbonne Paris Nord,
99 avenue Jean-Baptiste Clément,
93430 Villetaneuse,
France}
\email{wittenberg@math.univ-paris13.fr}



\begin{document}

\begin{abstract}
Starting from an Enriques surface over $\mathbb{Q}(t)$ constructed by Lafon, we give the first examples of smooth projective weakly special threefolds which fibre over the projective line in Enriques surfaces (resp.\ K3 surfaces) with nowhere reduced but non-divisible fibres and general type orbifold base. We verify that these families of Enriques surfaces (resp.\ K3 surfaces) are non-isotrivial and compute their fundamental groups by studying the behaviour of local points along certain \'etale covers.
The existence of these threefolds implies that the Weakly Special Conjecture formulated in 2000 contradicts the Orbifold Mordell Conjecture, and hence the abc conjecture. 
Using these examples, we can also easily disprove several complex-analytic analogues of the Weakly Special Conjecture. Finally, these examples  show  that Enriques surfaces and K3 surfaces can have non-divisible but nowhere reduced degenerations, thereby answering a question raised in 2005.
\end{abstract}
\maketitle

\thispagestyle{empty}

 \tableofcontents
 
\section{Introduction}

This article is concerned with new examples of weakly special varieties, a class of varieties first considered
in~\cite{CTSSD}:

\begin{definition}\label{def:ws}
A smooth projective variety $X$ over a field $k$ of characteristic zero with algebraic closure~$\overline{k}$ is \emph{weakly special} if no finite \'etale cover of~$X_{\overline{k}}$ rationally dominates a positive-dimensional smooth projective variety of general type over~$\overline{k}$. 
\end{definition}

As was noted in~\cite[Remark~3.2.1]{CTSSD},
Lang's conjecture on the non-density of rational points on varieties of general type \cite{Langconj}, together with the Chevalley--Weil theorem, predicts that a smooth projective variety $X$ over a number field $K$ for which $X(K)$ is dense is weakly special. One might wonder about the converse (see \cite[Conjecture~1.2]{HarrisTschinkel}):

\begin{conjecture}[Weakly Special Conjecture]\label{conj}
Let $X$ be a smooth projective variety over a number field~$K$. If $X$ is weakly special, then there is a finite extension $L/K$ such that $X(L)$ is dense.
\end{conjecture} 

Our main result disproves the Weakly Special Conjecture, assuming the Orbifold Mordell Conjecture (explained below, see Conjecture \ref{conj:orb_mor}). 

\begin{alphtheorem}\label{thm:ws_contradicts_abc}
The Weakly Special Conjecture contradicts the Orbifold Mordell Conjecture.
\end{alphtheorem}

The abc conjecture over number fields \cite[p.~84]{Vojta87} is known to imply the Orbifold Mordell Conjecture (see Theorem \ref{thm:abc_implies_orb_mordell} below), so that Theorem \ref{thm:ws_contradicts_abc} implies that the Weakly Special Conjecture contradicts abc.   For a summary of the relationships between the conjectures discussed here,  see the diagram at the end of the introduction.
 
A ``corrected''  version of the Weakly Special Conjecture has been known for some time and is recalled below (see Conjecture \ref{conj:special}). The latter conjecture replaces the hypothesis of $X$ being weakly special with the hypothesis of $X$ being special (Definition \ref{definition:special}).
The proof of Theorem~\ref{thm:ws_contradicts_abc}
precisely consists in constructing certain non-special weakly special threefolds.

Colliot-Thélène, Skorobogatov and Swinnerton-Dyer made in~\cite{CTSSD} the
key observation that the
 density of rational points on a variety fibred over the projective line can be obstructed by the presence of sufficiently many multiple
fibres, thanks to Faltings' theorem
(\emph{quondam} Mordell's conjecture).
The non-density,
assuming the Orbifold Mordell Conjecture,
of the rational points of the threefolds we construct in the proof of Theorem~\ref{thm:ws_contradicts_abc}
results from a refinement of this observation due to the second-named author,
who
introduced to this effect
two ways of attributing
multiplicities to nowhere reduced degenerations,
in~\cite[D\'efinition~2.1 and D\'efinition~4.2]{Ca11},
using either a gcd or an infimum.
  To explain this, let $f \colon X \to S$ be a     surjective morphism of integral noetherian schemes with $X$ normal and $S$ regular and let $s \in S$ be a point of codimension one with closure $D_s \subset S$. Let $X_{D_s}$ be the scheme-theoretic preimage of $D_s$ and note that $X_{D_s}$ is a non-empty divisor. We can write 
\[
X_{D_s} = R + \sum_{i=1}^n a_i F_i ,
\]
where $R$ is an effective $f$-exceptional divisor (i.e., none of its components surject onto $D_s$), each $F_i$ is an irreducible component of $X_{D_s}$ surjecting onto $D_s$, and $a_i\in \mathbb{Z}_{\geq 1}$. Note that $X_s = \sum_{i=1}^n a_i (F_i)_s$.

\begin{definition}
We say that the fibre $X_s$ is \emph{inf-multiple} if $\inf\{a_1,\ldots,a_n\} \geq 2$. We refer to this infimum as the \emph{inf-multiplicity} of the fibre. We say that the fibre $X_s$ is \emph{divisible} if $\gcd\{a_1,\ldots,a_n\} \geq 2$. We refer to this integer as the \emph{gcd-multiplicity} of the fibre.    
\end{definition}

Obviously, the gcd-multiplicity divides the inf-multiplicity. Moreover, the inf-multiplicity and gcd-multiplicity coincide if $X\to S$ is a family of abelian varieties or rationally connected varieties in characteristic zero (see \cite[Remarque~1.3]{Ca05}). However, by the main result of \cite{Winters}, there exist families of higher genus curves with non-divisible inf-multiple fibres; explicit examples of such families over $\mathbb{P}^1_{\mathbb{C}}$ are constructed in \cite[Th\'eor\`eme~5.1]{Ca05} and \cite{Stoppino}.  We note that the motivation for constructing such families was similar to the present one: the Orbifold Mordell Conjecture implies  in these examples   the  non-density of rational points on certain simply connected surfaces of general type fibred over $\mathbb{P}^1$ with specific inf-multiple, but no gcd-multiple,  fibres. 
%% The idea of using the Orbifold Mordell Conjecture to this end originated from the use of Faltings's theorem
%% (\emph{quondam} Mordell's conjecture) in \cite{CTSSD} to prove the non-density of rational points
%% on surfaces that are not weakly special and nevertheless do not dominate a variety of general type.

We prove Theorem \ref{thm:ws_contradicts_abc} by showing that there exists a threefold fibred in Enriques surfaces over the projective line with an inf-multiple, but non-divisible fibre. 
%(Recall that a smooth projective two-dimensional variety $S$ over a field $k$ is an \emph{Enriques surface (over $k$)} if $\omega_S$ is non-trivial $2$-torsion and, for every prime number $\ell$ invertible in $k$, the \'etale cohomology group $\mathrm{H}^2_{\et}(S,\mathbb{Q}_{\ell})$ is of dimension ten over $\mathbb{Q}_\ell$.)
The construction of this threefold relies on a subtle example due to Lafon \cite{Lafon} of an Enriques surface over $\mathbb{Q}(t)$ with no $\mathbb{C}\laurent{t}$-point (see Section~\ref{section:lafon}).
We note that this Enriques surface is given by explicit
equations which first appeared in~\cite{CTSSD}.
Lafon's motivation for constructing such an Enriques surface came from an old question of Serre (in a letter to Grothendieck \cite[p.~152]{GS}): If $X$ is a smooth projective variety over the function field $K$ of a complex curve $S$ with $\mathrm{H}^i(X,\mathcal{O}_X) = 0$ for all $i>0$, does $X$ necessarily have a $K$-point? Although Serre's question was first answered negatively by \cite{Graberetal}, Lafon's Enriques surface was the first explicit counterexample to Serre's question. 
 
Lafon's main result implies that the threefolds we obtain have precisely one inf-multiple fibre, which is of inf-multiplicity two. Then, by either appealing to the fact that Enriques surfaces over $\mathbb{C}\laurent{t}$ have index one (see e.g.~\cite[Theorem~1]{ELW}) or an explicit computation which shows that there is a component of multiplicity three, we can deduce that it has no divisible fibres; see Theorem \ref{thm:lafon}. As a consequence, we obtain the following result: 

\begin{alphtheorem}[Corollary \ref{corollary:first_theorem}] \label{thm1}
For every integer $d \geq 1$, there exists a smooth projective threefold $X_d$ over $\mathbb{Q}$ and a morphism $f\colon X_d \to \mathbb{P}^1_{\mathbb{Q}}$ such that the following properties hold.
\begin{enumerate}
\item The generic fibre of $X_d \to \mathbb{P}^1_{\mathbb{Q}}$ is an Enriques surface over $\mathbb{Q}(t)$.
\item There are exactly $d$ points in $\mathbb{P}^1(\overline{\mathbb{Q}})$ over which the fibre of $f$ is inf-multiple, and all of these fibres have inf-multiplicity two and are non-divisible.
\end{enumerate}
\end{alphtheorem}

By investigating how local points above $t=0$ behave along a certain K3 double cover of Lafon's Enriques surface, we obtain a similar result for K3 surfaces. 

\begin{alphtheorem}[Consequence of Theorem \ref{thm2:K3_version}] \label{thm1:K3_version}
For every integer $d \geq 1$, there exists a smooth projective threefold $Y_d$ over $\mathbb{Q}$ and a morphism $g\colon Y_d \to \mathbb{P}^1_{\mathbb{Q}}$ such that the following properties hold.
\begin{enumerate}
\item The generic fibre of $Y_d \to \mathbb{P}^1_{\mathbb{Q}}$ is a K3 surface over $\mathbb{Q}(t)$.
\item There are exactly $d$ points in $\mathbb{P}^1(\overline{\mathbb{Q}})$ over which the fibre of $g$ is inf-multiple, and all of these fibres have inf-multiplicity two and are non-divisible.
\end{enumerate}
\end{alphtheorem}

The threefolds $X_d$ in Theorem \ref{thm1} are (automatically) weakly special by the following result, which we prove using the aforementioned fact that the index of an Enriques surface over $\mathbb{C}\laurent{t}$ equals one:

\begin{alphtheorem}[Consequence of Corollary \ref{cor:enriques_fibration}]\label{thm:enriques_fibration}
Let $X\to C$ be a surjective morphism of smooth projective varieties over a field $k$ of characteristic zero with $C$ a smooth projective curve of genus at most one over $k$. If the generic fibre of $X\to C$ is an Enriques surface, then $X$ is weakly special.
\end{alphtheorem}

The threefolds $X_d$ constructed in Theorem~\ref{thm1} are also,
for $d\geq 5$,
the first examples of ``gcd-special'' nonspecial varieties, thereby
answering a question formulated in \cite[p.624]{Ca04}.

Lafon's Enriques surface directly leads to the construction of the threefold $X_1$ in Theorem \ref{thm1}. We then obtain the threefold $X_d$ by pulling back the fibration $X_1\to \mathbb{P}^1$ along a suitable degree $d$ morphism $\mathbb{P}^1\to \mathbb{P}^1$, thereby increasing the number of inf-multiple fibres arbitrarily.

The threefolds $Y_d$ in Theorem \ref{thm1:K3_version} are also weakly special; this does not however follow directly from Theorem \ref{thm:enriques_fibration}, but is more specific to the construction (see Section \ref{section:K3}). 

To prove Theorem \ref{thm:ws_contradicts_abc}, we first interpret $K$-rational points on $X_d$ as $\mathcal{O}_{K,S}$-integral points on some model. By \cite{Ca04, Ca11}, these integral points project down to integral points on $\mathbb{P}^1_{\mathcal{O}_K}$ satisfying certain tangency conditions determined by the inf-multiple fibres of $X_d\to \mathbb{P}^1$. To make this more precise, we employ the notions of orbifold base and orbifold morphism.

Let $X$ be an integral locally factorial noetherian scheme, and let $\Delta = \sum_i (1-\frac{1}{m_i}) D_i$ be a $\mathbb{Q}$-divisor, where $m_i\in \mathbb{Z}_{\geq 1}\cup \{\infty\}$. We will refer to $(X,\Delta)$ as an \emph{orbifold}, and refer to $m_i$ as the \emph{multiplicity} of $D_i$ in $\Delta$.

If $X$ is a locally factorial variety over a field $k$ of characteristic zero and $(X,\Delta)$ is an orbifold, then we say that $(X,\Delta)$ is \emph{smooth} if $X$ is smooth and the support of $\Delta$ has simple normal crossings. Moreover, we say that the orbifold $(X,\Delta)$ is \emph{proper} if $X$ is proper over $k$. If $(X,\Delta)$ is a smooth proper orbifold, then we say that $(X,\Delta)$ is of \emph{general type} if $K_X+\Delta$ is a big $\mathbb{Q}$-divisor. 

An important class of orbifolds are those associated to a rational fibration (see \cite[Definition~4.2]{Ca11}). More precisely, let $f \colon X \ratmap Y$ be a dominant rational map of smooth proper varieties over $k$ with geometrically connected fibres and choose a proper birational morphism $X'\to X$ with $X'$ a smooth projective variety such that $X' \to X \ratmap Y$ is a morphism.
For a prime divisor $D\subset Y$ with generic point $\eta_D$, let $m_f(D)$ be the inf-multiplicity of $X'\to Y$ over $\eta_D$ (this positive integer is independent of the choice of $X'\to X$ by Remark \ref{remark:inf_multiplicity}). 
 We define the \emph{orbifold divisor} $\Delta_f$ \emph{of $f$} to be 
\[ \Delta_f := \sum_{D\subset Y} \left(1-\frac{1}{m_f(D)}\right) D, \]
where the sum runs over all prime divisors $D$ of $Y$ and refer to $(Y,\Delta_f)$ as the \emph{orbifold base of}~$f$. 

With notation as in Theorem \ref{thm1}, the orbifold base of $X_d\to \mathbb{P}^1_{\mathbb{Q}}$ is $(\mathbb{P}^1_{\mathbb{Q}}, \frac{1}{2}[a_1]+\ldots +\frac{1}{2}[a_d])$ for some $\mathrm{Gal}(\overline{\mathbb{Q}}/\mathbb{Q})$-invariant collection of pairwise distinct algebraic numbers $a_1,\ldots,a_d$. Computing the degree of the canonical divisor shows that the orbifold base of $X_d\to \mathbb{P}^1_{\mathbb{Q}}$ is of general type precisely when $d\geq 5$. 

To see the number-theoretic implications of the inf-multiple fibres we need the notion of orbifold integral point (or more generally orbifold morphism). To define this, for $f \colon X \to Y$ a morphism of integral schemes and $D$ a Cartier divisor on $Y$ with $f(X) \not \subset D$, we let $f^\ast D$ be the pullback Cartier divisor on $X$; see \cite[Tag~02OO]{stacks-project}. 

We start with the general notion of an orbifold morphism. 
Let $T$ be an integral normal noetherian scheme. An \emph{orbifold morphism} $f \colon T \to (X,\Delta)$ is a morphism of schemes $f \colon T \to X$ such that $f(T) \not\subset \supp(\Delta)$ and, for every $i$ and every irreducible component $E$ of the scheme-theoretic inverse image $f^{-1} (D_i)$, the multiplicity of $E$ in $f^* D_i$ is at least $m_i$. 

To apply the above notion of orbifold morphism in the context of integral points, we define the notion of a model of an orbifold as follows. 
Let $S$ be an integral noetherian scheme with function field $K=K(S)$. 
Let $(X,\sum_{i}(1-\frac{1}{m_i})D_i)$ be an orbifold, where $X$ is a finite type separated $K$-scheme. A \emph{model for $(X,\Delta)$ over $S$} is the data of an orbifold $(\mathcal{X}, \sum_{i} (1-\frac{1}{m_i}) \mathcal{D}_i)$, where $\mathcal{X}$ is an integral locally factorial flat finite type separated $S$-scheme and $\mathcal{D}_i$ is a divisor on $\mathcal{X}$, together with an isomorphism $\varphi \colon \mathcal{X}_K\to X$ over $K$ such that $\varphi(\mathcal{D}_{i,K})$ on $X$ equals $D_i$.

We now specialize to number fields. Let $(X,\Delta)$ be an orbifold with $\Delta = \sum_i (1-\frac{1}{m_i})D_i$, where $X$ is a locally factorial variety over a number field $K$. We say that $(X,\Delta)$ is \emph{Mordellic} if, for every finite field extension $L/K$, every regular $\mathbb{Z}$-finitely generated subring $R \subset L$ and every model $(\mathcal{X}, \sum_i (1-\frac{1}{m_i})\mathcal{D}_i)$ for $(X_L, \sum_{i} (1-\frac{1}{m_i})D_{i,L})$ over $R$, the set $(\mathcal{X}, \sum_i (1-\frac{1}{m_i}) \mathcal{D}_i)(R)$ of orbifold maps $\Spec R \to (\mathcal{X}, \sum_i (1-\frac{1}{m_i})\mathcal{D}_i)$ is finite. 
We say that $(X,\Delta)$ is \emph{arithmetically special} (or: \emph{satisfies potential density}) if there is a finite field extension $L/K$, a regular $\mathbb{Z}$-finitely generated subring $R \subset L$ and a model $(\mathcal{X}, \sum_i (1-\frac{1}{m_i}) \mathcal{D}_i)$ for $(X_L, \sum_i (1-\frac{1}{m_i})D_{i,L})$ over $R$ such that $(\mathcal{X}, \sum_i (1-\frac{1}{m_i})\mathcal{D}_i)(R)$ is dense in $X$.
 
It is not hard to show that if $(X,\Delta)$ is a smooth proper orbifold curve over a number field $K$ and $K_X+\Delta$ is not big, then $(X,\Delta)$ is arithmetically special \cite{Ca05}. The converse to this statement is the Orbifold Mordell conjecture (which first appeared in \cite[Conjecture~4.5]{Ca05} and is a very special case of \cite[Conjecture~13.23]{Ca11}):

\begin{conjecture}[Orbifold Mordell]\label{conj:orb_mor}
If $(X,\Delta)$ is a smooth proper orbifold curve of general type over a number field $K$, then $(X,\Delta)$ is Mordellic.
\end{conjecture}

If $X$ is of genus at least two or $\Delta$ only has infinite multiplicities, then the Orbifold Mordell Conjecture follows from Faltings's finiteness theorem \cite{FaltingsComplements}. However, the Orbifold Mordell Conjecture is open for every orbifold curve $(X,\Delta)$ of general type such that $\Delta$ only has finite multiplicities and the genus of $X$ is at most one. 

We note that the abc conjecture implies Conjecture \ref{conj:orb_mor}; see \cite[Theorem~5.3]{Smeets} (which builds on \cite{AbramovichBirGeom} and \cite[Remarque~4.8]{Ca05}, and generalizes Elkies' theorem that abc implies Mordell \cite{Elkies}).

\begin{theorem}[abc implies Orbifold Mordell]\label{thm:abc_implies_orb_mordell}
Assume that the abc conjecture over number fields holds. Then the Orbifold Mordell Conjecture holds.
\end{theorem}

With notation as in Theorem \ref{thm1}, since $X_d$ is weakly special (Theorem \ref{thm:enriques_fibration}), the Weakly Special Conjecture predicts that, for every $d \geq 1$, the rational points on the threefold $X_d$ in Theorem~\ref{thm1} are potentially dense (i.e., $X_d$ is arithmetically special). However, choosing $d \geq 5$, this contradicts the Orbifold Mordell Conjecture, and hence abc, because a dense set of $K$-rational points on $X_d$ for some number field $K$ induces a dense (hence infinite) set of $\mathcal{O}_{K,S}$-integral orbifold points on some suitable model of the (general type!) orbifold base of $X_d\to \mathbb{P}^1$ for some finite set of finite places $S$ of $K$:

\begin{alphtheorem}[Proved in Section \ref{section:main_results}]\label{thm:abc}
Assume the Orbifold Mordell Conjecture (Conjecture \ref{conj:orb_mor}). Then, for every $d\geq 5$, the smooth projective threefold $X_d$ from Theorem \ref{thm1} is weakly special but, for every number field $K$, the set $X_d(K)$ is not dense.
\end{alphtheorem}

Actually, to disprove the Weakly Special Conjecture, we ``only'' need the existence of an integer $d \geq 5$ and algebraic numbers $a_1,\ldots,a_d$ such that the Orbifold Mordell Conjecture holds for the orbifold $(\mathbb{P}^1, \frac{1}{2} [a_1] + \ldots + \frac{1}{2} [a_d])$; see Theorem \ref{thm2} (or Remark \ref{remark:strategy}) for a more precise statement. 

  Since the Weakly Special Conjecture might be false, it is natural to wonder about alternative conjectures. The more plausible conjecture was stated in \cite{Ca04}. To explain this conjecture, let $(X,\Delta)$ be a proper orbifold over a field $k$ of characteristic zero and write $\Delta = \sum_{j \in J} (1-\frac{1}{m_j}) D_j$ (recall that this in particular means that $X$ is a locally factorial variety, so that the $D_j$ are Cartier divisors). A smooth proper orbifold $(X',\Delta')$ is called a \emph{smooth proper model of $(X,\Delta)$}  if there is a proper birational morphism $\psi\colon X'\to X$ with $\psi_\ast \Delta' = \Delta$ such that $(X',\Delta') \to (X,\Delta)$ is ``orbifold'', i.e., for every $j\in J$ and every prime divisor $E$ on $X'$ with $\psi(E) \subseteq D_j$, the coefficient of $E$ in $\psi^*D_j$ is at least $m_j/m_E$, where $m_E$ is the multiplicity of $E$ in $\Delta'$ if $E$ appears in $\Delta'$ and $1$ otherwise. (Note that every proper orbifold over a field of characteristic zero has a smooth proper model by Hironaka's resolution of singularities.) If $(X,\Delta)$ is a proper orbifold, then we say that $(X,\Delta)$ is of \emph{general type} if some smooth proper model $(X',\Delta')$ of $(X,\Delta)$ is of general type (i.e., $K_{X'} + \Delta'$ is big).
Let $f \colon X \ratmap Y$ be a dominant rational map of smooth proper varieties with geometrically connected fibres. 
We say that $f$ is a \emph{fibration of general type} if, for every proper birational morphism $Y'\to Y$ with $Y'$ smooth, the orbifold  base  $(Y',\Delta_{f'})$ of $f'\colon X\ratmap Y'$    is of general type.

\begin{definition}\label{definition:special}  
A smooth proper variety $X$ over $k$ is \emph{special} if $X_{\overline{k}}$ has no fibrations of general type onto a positive-dimensional smooth proper variety.  A proper variety $X$ over $k$ is \emph{special} if some (hence any) smooth proper model $X'$ of $X$ is special. 
\end{definition}

The following conjecture   (see \cite[Conjecture~9.20]{Ca04}) replaces the hypothesis   of ``being weakly special'' in Conjecture \ref{conj} to ``being special''; note that the term weakly special was introduced in \cite{Ca05} precisely to stress the difference between the two notions.

\begin{conjecture}\label{conj:special}
Let $X$ be a smooth projective variety over a number field $K$. Then $X$ is special if and only if there is a number field $L/K$ such that $X(L)$ is dense. 
\end{conjecture}

Smooth projective rationally connected varieties are special \cite[Theorem~3.22]{Ca04}. Furthermore, abelian varieties, K3 surfaces, and Enriques surfaces are special. In fact, all smooth projective varieties of Kodaira dimension zero are special; see \cite[Theorem~5.1]{Ca04}. If $X \to Y$ is a finite \'etale morphism of smooth projective varieties, then $X$ is special if and only if $Y$ is special \cite[Theorem~5.12]{Ca04}. Therefore, any special smooth projective variety is weakly special. 

The above conjecture is open for rationally connected varieties and general K3 surfaces, but is known for elliptic K3 surfaces, abelian varieties and Enriques surfaces \cite{HassettSurvey, HassettTschinkel, BT1, BT2, BT3}. 

The notion of a special variety is intimately related to the notion of the ``core map'' of a variety. Let $X$ be a smooth projective variety over $\mathbb{C}$. In \cite{Ca04} it is shown that there is a unique almost holomorphic fibration of general type $f\colon X\ratmap Y$ such that the general fibre of $f$ is a special variety. The orbifold base $(Y,\Delta_f)$ of this fibration is called the \emph{core} of $X$, and it is unique up to birational modification of $(Y,\Delta_f)$. Observe that $X$ is special if and only if its core is trivial.  The core   map splits $X$ into two parts with opposite geometry: the fibres (special) and the orbifold base (general type).  This splitting of $X$ is conjecturally also   reflected in the arithmetic and complex-analytic properties of $X$.

In \cite{BT}, Bogomolov--Tschinkel construct examples of weakly special threefolds. These threefolds do not have a dense set of rational points over any number field, assuming Vojta's higher-dimensional abc conjecture \cite{VojtaIMRN}; this can be deduced from \cite[Corollary~3.3]{AbramovichVA} and the fact that the Bogomolov--Tschinkel threefolds come with an elliptic fibration whose orbifold base is of general type. 
However, the non-density of rational points on their threefolds (and the examples constructed afterwards in \cite{CP07} and \cite{RTW})  seemingly cannot be proved,  assuming only the abc conjecture. Indeed, these threefolds all have a two-dimensional core (see \cite[Definition~3.1]{Ca04}), so that the abc conjecture does not seem to suffice to prove the nondensity of integral points on the relevant orbifold base. The advantage of the threefolds $X_d$ constructed in Theorem \ref{thm1} is that their core is one-dimensional for every $d\geq 5$.

Since K3 surfaces and Enriques surfaces are special \cite[Theorem~5.1]{Ca04}, Theorems \ref{thm1} and \ref{thm1:K3_version} show that the specialization of a smooth projective special variety in a flat family can be inf-multiple but non-divisible (i.e., in the terminology of \cite{Ca05}, the classical orbifold base and the non-classical orbifold base do not necessarily agree), thereby answering a question raised in \cite[Question~1.4]{Ca05}. (Since inf-multiple degenerations of rationally connected varieties do not exist and those of abelian varieties must be divisible \cite[Remarque~1.3]{Ca05}, our examples in K3 surfaces and Enriques surfaces are the ``smallest'' possible in a sense.)
%See also Théorème 7.7 in Colliot-Thélène's lecture notes ``Variétés presque rationnelles, leurs points rationnels et leurs dégénérescences''

Complex-analytic and function field analogues of the Weakly Special Conjecture were considered and disproved in \cite{BJR, CP07, CWKoba, RTW}. These disproofs were based on variants of the construction in \cite{BT} and are quite difficult, as the core of the weakly special threefolds in all these examples is two-dimensional. Since the hyperbolicity of orbifold curves of general type is considerably easier to prove and our weakly special threefolds $X_d$ have a one-dimensional core for $d\geq 5$, we can give much simpler disproofs; see Theorem \ref{thm:ws_analytic} for a precise statement. 

In Section \ref{section:fundamental_group} we show that Lafon's Enriques surface is non-isotrivial (see Theorem \ref{thm:lafon_is_non_isotrivial}).
We then determine the fundamental group of the threefolds we constructed; see Theorem \ref{thm:lafon_pi1} for a precise statement. Most notable here is the route we take to prove non-isotriviality of Lafon's Enriques surface. We use non-liftability of local points to conclude that a K3 double cover never extends to a finite \'etale cover of any regular projective model (which implies non-isotriviality). However, quite surprisingly, the issue here is not the inf-multiple fibre at $t=0$. In fact, in Section \ref{section:K3}, we prove that there is a K3 double cover of Lafon's Enriques surface which extends to a finite \'etale cover near $t=0$; this is how we prove Theorem \ref{thm1:K3_version}.

 For the reader’s convenience, we summarize in the diagram below all (possibly conjectural) implications among the relevant properties and conjectures.

\bigskip
\begin{tikzpicture}[>=Implies]
  \node[align=center] (Z) at (0,0) {rational points on $X$\\are potentially dense};
  \node (WS) at (10,0) {$X$ is weakly special};
  \node (S) at (5.5,0) {$X$ is special};
  \node (L) at (5,2.5) {Lang's conjecture};

  \draw[double distance=2pt, ->] (Z) .. controls ($(Z)+(1,2)$) and ($(WS)+(-1,2)$) .. (WS);
  \draw[double distance=2pt, ->] (S) to (WS);
  \draw[double distance=2pt, <->, dash pattern=on 6pt off 2pt] (Z) -- node[midway, above] {{\tiny Conjecture~\ref{conj:special}}} (S);
  \draw[double distance=2pt, <-, bend right=25, postaction={decorate, decoration={markings, mark= at position 0.5 with {\node[scale=2] {$\times$};}}}] (S) to node[midway,below,yshift=-3pt] {\tiny \cite{BT}} (WS);
  \draw[double distance=2pt, ->] ($(L.south)+(0,0.1)$) -- (5,1.75);

  \draw[double distance=2pt, <-, name path=F] (Z) .. controls ($(Z)+(1,-2)$) and ($(WS)+(0,-2)$) .. (WS);
  \path[name path=V] (5,-3) -- (5,1);
  \path [name intersections={of=F and V, by=I}];
  \node at (I) [scale=2] {$\times$};
  \node (O) at ($(I)+(0,-1)$) {Orbifold Mordell conjecture};
  \draw[double distance=2pt, ->] ($(O.north)+(0,-0.045)$) -- (5,-1.75);
  \node (ABC) at ($(O)+(5,0)$) {abc conjecture};
  \draw[double distance=2pt, ->] (ABC) to (O);
\end{tikzpicture}

\begin{con}
A variety over a field $k$ is a geometrically integral finite type separated $k$-scheme. 
\end{con}

\begin{ack}
We thank the referee for their helpful comments.
The third-named author is most grateful to Erwan Rousseau for many inspiring discussions on special varieties.
\end{ack}

\section{Weakly special varieties}

The goal of this section is to show that the total space of a family of weakly special varieties over a weakly special base with no divisible fibres in codimension one is itself weakly special (see Theorem \ref{thm:ws} below for a precise statement). 
 
\begin{lemma}\label{lemma:stein}   
Let $X \to Y$ be a proper surjective morphism of integral regular noetherian schemes over~$\mathbb{Q}$. Suppose that for every point $y$ of codimension one and every connected component $F$ of $X_y$ the divisor $F$
on $X \times_Y \Spec(\mathcal{O}_{Y,y})$
is not divisible. Then, the Stein factorization $Y' \to Y$ of $X \to Y$ is finite \'etale.  
\end{lemma}
\begin{proof}
Note that $Y' \to Y$ is finite surjective and that $Y'$ is an integral normal noetherian scheme. Therefore, by Zariski-Nagata purity \cite[Theorem~XI.3.1]{SGA1}, we may and do assume that $Y = \Spec R$ is the spectrum of a discrete valuation ring $R$. Suppose that $y'$ is a point of $Y'$ lying over the closed point $y$ of $Y$ with ramification index $e$. We have to show that $e=1$.  To do so, let $E \subset X$ be the connected component of $X_y$ lying over $y'$. (Note that $E$ exists.) Observe that $E$, as a divisor, equals $e \cdot X_{y'}$ where $X_{y'}$ is the scheme-theoretic fibre of $X \to Y'$ over $y'$. Since $E$ is not divisible (by assumption), it follows that $e=1$. 
\end{proof}
 
\begin{remark} 
Note that Lemma \ref{lemma:stein} follows from \cite[Tag~0BUN]{stacks-project} if the fibres of $X\to Y$ are geometrically reduced. The fact that Lemma \ref{lemma:stein} holds under the weaker assumption that $X \to Y$ has no divisible fibres in codimension one suggests that non-divisible fibres are quite similar to geometrically reduced (or even smooth) fibres. For example, at the level of fundamental groups, a fibration with non-divisible fibres induces a homotopy exact sequence (similar to the exact sequence associated to a smooth fibration); see Lemma \ref{lemma:nori} for a precise statement. 
\end{remark}

\begin{theorem} \label{thm:ws}  
Let $X$ and $Y$ be smooth projective varieties over an algebraically closed field $k$ of characteristic zero. Let $X \to Y$ be a surjective morphism with connected fibres. Suppose that for every $y$ in $Y$ of codimension one, the fibre $X_y$ is not divisible, that $Y$ is weakly special, and that for all $y$ in $Y(k)$ with $X_y$ smooth, the smooth projective variety $X_y$ is weakly special. Then, $X$ is weakly special. 
\end{theorem}
\begin{proof}
Let $X'\to X$ be a finite \'etale morphism with $X'$ a smooth projective variety over $k$. Let $X' \to Y' \to Y$ be the Stein factorization of $X'\to Y$. By Lemma \ref{lemma:stein}, the morphism $Y' \to Y$ is finite \'etale. In particular, since $Y$ is weakly special, it follows that $Y'$ is weakly special. Moreover, the smooth fibres of $X'\to Y'$ are smooth projective weakly special varieties (as each such fibre is a finite \'etale cover of some smooth fibre of $X\to Y$). To conclude the proof, it suffices to show that, given a positive-dimensional smooth projective variety $Z$ of general type and a rational map $g \colon X' \dashrightarrow Z$, the map $g$ is not dominant. 

Assume that $g$ were dominant. If $g$ contracts the general fibre of $X'\to Y'$, it (rationally) factors over $Y'\ratmap Z$. Since $Y'$ is weakly special (as noted above), the map $Y' \ratmap Z$ is not dominant, hence $g$ is not dominant either. Thus, $g$ does not contract the general fibre. Then, the images $g(X'_{y'})$ with $y'$ a general point of $Y'$ form a family of positive-dimensional subvarieties of $Z$ covering $Z$. Since $Z$ is of general type, the general member $g(X'_{y'})$ of this family is of general type.
Thus, the weakly special variety $X'_{y'}$ dominates the positive-dimensional variety $g(X'_{y'})$ of general type, a contradiction. 
\end{proof}

\begin{remark} 
We will apply Theorem \ref{thm:ws} to the case that the fibres of $X \to Y$ are Enriques surfaces. (As we recall in Proposition \ref{prop:enriques_fibration}, a fibration of Enriques surfaces cannot have a divisible fibre.) To do so, we note the simple observation that Enriques surfaces (and K3 surfaces) are weakly special. Indeed, an Enriques surface only has a single nontrivial connected étale covering, namely its K3 double cover. This K3 double cover is simply connected and hence cannot dominate a curve of general type. Moreover, a K3 surface has Kodaira dimension zero and hence cannot dominate a surface of general type, either.
% (This follows from the fact that an Enriques surface has a finite fundamental group, hence does not dominate a curve of general type, and that an Enriques surface has Kodaira dimension zero, and therefore does not dominate a surface of general type.)
\end{remark}

\begin{remark}\label{remark:K3_construction}  
The assumption that $X \to Y$ has no divisible fibres in Theorem \ref{thm:ws} is essential, even if the smooth fibres are simply connected. Indeed, let $E$ be a K3 surface with a fixed-point free involution $\tau \colon E \to E$, and let $D$ be a genus two curve with hyperelliptic involution $\sigma \colon D \to D$. Note that $(\sigma,\tau) \colon D \times E \to D \times E$ is a fixed-point free involution. Let $X = (D\times E)/(\sigma,\tau)$ and note that the projection $D\times E\to D$ induces a surjective morphism $X\to D/\sigma \cong \mathbb{P}^1$ whose generic fibre is a K3 surface. However, $X$ is not weakly special, as its finite \'etale cover $D\times E$ dominates the positive-dimensional variety $D$ of general type.
\end{remark}

\section{The index and divisible fibres} 
In this section, we use the index of a variety to show that a family of Enriques surfaces cannot have divisible fibres (Proposition \ref{prop:enriques_fibration}).
As a consequence we deduce that the total space of a family of Enriques surfaces over a curve of genus at most one is weakly special (Corollary \ref{cor:enriques_fibration}).

In our proof of Theorem \ref{thm:main_insight} below, we will use the following lemma on multiplicities of fibres of regular models.

\begin{lemma}\label{lemma:local_points}
Let $X$ be a smooth proper connected scheme over $\mathbb{C} \laurent{t}$ and let $\mathcal{X}$ be a regular proper model over $\mathbb{C} \powerseries{t}$. 
Write $\mathcal{X}_0 = \sum_i a_i F_i$ for the decomposition of $\mathcal{X}_0$ into prime divisors, and let $a = \inf(a_i)$. Then the following statements hold.
\begin{enumerate}
\item For every $i$, the set $X(\mathbb{C} \laurent{t^{1/a_i}}) $ is non-empty. 
\item If $m\geq 1$ is a positive integer such that $X(\mathbb{C}\laurent{t^{1/m}})$ is non-empty, then $m$ is a $\mathbb{Z}_{\geq 0}$-linear combination of the $a_i$.
\item If $a < m < 2a$ and $X(\mathbb{C}\laurent{t^{1/m}})$ is non-empty, then $\mathcal{X}_0$ has an irreducible component of multiplicity $m$. 
\item In particular, if $a=2$ and $X(\mathbb{C}\laurent{t^{1/3}}) \neq \emptyset$, then $\mathcal{X}_0$ has an irreducible component of multiplicity $3$ and is thus not divisible. 
\end{enumerate}
\end{lemma}
\begin{proof} 
We note that $(i)$ is a special case of \cite[\S 9.1, Corollary~9]{BLR} (see also \cite[Lemma~2.4]{KestelootNicaise} or the proof of \cite[Lemme~4.6]{Wit15}); we stress that $X$ is not necessarily geometrically connected over $K$. To prove $(ii)$, let $x \in X(\mathbb{C}\laurent{t^{1/m}})$. Let $\epsilon_x \in \mathcal{X}(\mathbb{C}\powerseries{t^{1/m}})$ be the corresponding multisection of $\mathcal{X}\to \Spec \mathbb{C}\powerseries{t}$. Now note that
\[ m = \epsilon_x \cdot \mathcal{X}_0 = \epsilon_x \cdot \sum_i a_i F_i = \sum_i (\epsilon_x \cdot F_i) a_i \]
exhibits $m$ as a $\mathbb{Z}_{\geq 0}$-linear combination of the $a_i$, as required. Note that $(iii)$ follows directly from $(ii)$ and that $(iv)$ is a special case of $(iii)$. 
\end{proof}

\begin{remark}
In the situation of the above lemma, under the additional assumption that the reduced special fibre $(\mathcal{X}_0)_{\red}$ has simple normal crossings, one can give a complete answer to the question of when $X(\mathbb{C} \laurent{t^{1/m}})$ is non-empty: Denote by~$I$ the set of irreducible components of~$\mathcal{X}_0$. Then, given an integer $m \in \mathbb{Z}$, the set $X(\mathbb{C}\laurent{t^{1/m}})$ is non-empty if and only if there exists a subset $J \subseteq I$ for which $\bigcap_{i \in J} F_i$ is non-empty and such that $m$ is a $\mathbb{Z}_{>0}$-linear combination of the~$a_i$ for $i \in J$. We will however not use this characterization in the remainder of the paper and refer to \cite[Proposition~2.5]{KestelootNicaise} for a proof.
\end{remark}

\begin{remark}[The inf-multiplicity] \label{remark:inf_multiplicity}
Let $X\ratmap Y$ be a dominant rational map of smooth projective varieties over a field $k$ of characteristic zero.
Let $X_1\to X$ and $X_2\to X$ be proper birational morphisms with $X_1$ and $X_2$ smooth projective varieties such that $X_1\to X\ratmap Y$ and $X_2\to X\ratmap Y$ are (surjective) morphisms.
Let $y\in Y$ be a point of codimension one. We claim that the inf-multiplicity $m_1$ of $(X_1)_y$ equals the inf-multiplicity $m_2$ of $(X_2)_y$. Indeed, to compute $m_1$ and $m_2$, we may replace $Y$ by $\Spec \mathbb{C}\powerseries{t}$. Then, $m_1$ is the smallest positive integer $a$ for which $X_1$ has a $\mathbb{C}\laurent{t^{1/a}}$-point. Since the generic fibre of $X_1\to Y$ is birational to the generic fibre of $X_2\to Y$ and since $X_1$ is smooth over $\mathbb{C}\laurent{t}$ and $X_2$ is proper over $\mathbb{C}\laurent{t}$, it follows from Lang--Nishimura \cite[Theorem~3.6.11]{PoonenRat} that $X_2$ (also) has a $\mathbb{C}\laurent{t^{1/m_1}}$-point. This shows that $m_2\leq m_1$. By symmetry, $m_1=m_2$, as claimed.
\end{remark}

\begin{definition}\label{definition:index}
Let $X$ be a variety over a field $K$. We define the \emph{index of $X$}, denoted by $i(X)$, to be the greatest common divisor of all natural numbers $n$ such that there exists a field extension $L/K$ of degree $n$ with $X(L) \neq \emptyset$.
\end{definition}

\begin{remark}[Index of an Enriques surface]\label{remark:enriques}
Any Enriques surface over $\mathbb{C}\laurent{t}$ has index one. Indeed, more generally, any smooth projective variety $X$ over $\mathbb{C}\laurent{t}$ with $\vert \chi(X,\mathcal{O}_X) \vert = 1$ (e.g., an Enriques variety as defined in \cite{Boissiere} or a smooth proper curve of genus two) has index one by \cite[Theorem~1]{ELW}. 
\end{remark} 

\begin{lemma}\label{lemma:index}
Let $X$ be a smooth proper variety over $\mathbb{C} \laurent{t}$ and let $\mathcal{X}$ be an integral regular proper model over $\mathbb{C} \powerseries{t}$. Write $\mathcal{X}_0 = \sum_i a_i F_i$ with $a_i\in \mathbb{Z}_{\geq 1}$ and $F_i$ an irreducible component of $\mathcal{X}_0$. Then the gcd of the $a_i$ is equal to the index of~$X$ and the infimum of the $a_i$ is equal to the minimal integer $m$ such that $X(\mathbb{C}\laurent{t^{1/m}})\neq \emptyset$.
\end{lemma}
\begin{proof}
These statements are elementary consequences of Lemma \ref{lemma:local_points}.
\end{proof}

In particular, the previous lemma implies that both the inf-multiplicity and the gcd-multiplicity of the special fibre can be read off of the generic fibre; note that the following proposition fails for K3 surfaces (Remark \ref{remark:K3_construction}).

\begin{proposition}\label{prop:enriques_fibration}
Let $X \to Y$ be a proper surjective morphism of integral normal noetherian schemes over~$\mathbb{Q}$. Suppose that $X$ is regular and that the generic fibre of $X \to Y$ is an Enriques surface. Then, the fibration $X \to Y$ has no divisible fibres in codimension one. 
\end{proposition}
\begin{proof}
Let $y \in Y$ be a point of codimension one. To prove the proposition, we may replace $Y$ by $\Spec \mathcal{O}_{Y,y}$. In addition, we may replace $\mathcal{O}_{Y,y}$ by its completion and strict henselization. Since every residue field of $Y$ is of characteristic zero, it follows that $Y = \Spec k \powerseries{t}$ for some algebraically closed field $k$ of characteristic zero. We may assume that $k=\mathbb{C}$. The result now follows from the fact that an Enriques surface over $\mathbb{C}\laurent{t}$ has index one (Remark \ref{remark:enriques}) and the fact that the gcd of the multiplicities of the special fibre of $X \to Y$ divides the index (Lemma \ref{lemma:index}). 
\end{proof}

Note that Theorem \ref{thm:enriques_fibration} from the introduction is a special case of the following more general statement.

\begin{corollary}\label{cor:enriques_fibration} 
Let $X \to Y$ be a surjective morphism of smooth projective varieties over a field $k$ of characteristic zero. If $Y$ is weakly special (e.g., a smooth projective curve of genus at most one) and the generic fibre of $X \to Y$ is an Enriques surface, then $X$ is weakly special. 
\end{corollary}
\begin{proof}
We may assume that $k$ is algebraically closed. Since the generic fibre of $X \to Y$ is an Enriques surface, it follows from Proposition \ref{prop:enriques_fibration} that $X \to Y$ has no divisible fibres in codimension one. Note that every smooth fibre of $X \to Y$ is an Enriques surface. In particular, since Enriques surfaces are weakly special (Remark \ref{remark:enriques}) and $Y$ is weakly special (by assumption), the result follows from Theorem \ref{thm:ws}. 
\end{proof}

We emphasize that Corollary \ref{cor:enriques_fibration} demonstrates a key element of our strategy. Namely, we will construct a threefold $X$ and a surjective morphism $X\to \mathbb{P}^1_{\mathbb{Q}}$ whose generic fibre is an Enriques surface and which has at least one inf-multiple fibre (Theorem \ref{thm:lafon}). Such a threefold $X$ is weakly special by Corollary \ref{cor:enriques_fibration}. 

\begin{remark}[Specialization index]\label{remark:spec_index} 
Let $X$ be a smooth proper variety over $\mathbb{C}\laurent{t}$ and $\mathcal{X}$ be a proper regular model over $\mathbb{C}\powerseries{t}$ such that the reduced subscheme underlying the special fibre has simple normal crossings. (Such a model exists by Hironaka's resolution of singularities.) Assume that $\mathrm{H}^i(X,\mathcal{O}_X)=0$ for all $i>0$ (e.g., $X$ is an Enriques surface over $\mathbb{C}\laurent{t}$). Then, by \cite[Theorem~4.6]{KestelootNicaise}, the ``specialization index'' of $X$ equals one, i.e., there is a closed point of the special fibre $\mathcal{X}_0$ such that the gcd of the multiplicities of the irreducible components containing $x$ equals $1$. 
\end{remark}

\section{Lafon's Enriques surface}\label{section:lafon}

In \cite{Lafon} Lafon constructed an explicit counterexample to a question of Serre (which was first answered negatively using indirect methods by Graber--Harris--Mazur--Starr \cite{Graberetal}). Indeed, inspired by the methods and equations introduced in \cite{CTSSD}, Lafon proved the following result. 

\begin{theorem}[Lafon]\label{thm:lafon}
The smooth surface in $\mathbb{A}^4_{\mathbb{Q}(t)}$ with coordinates $(u,x,y,z)$ given by 
\begin{eqnarray*}
x^2 - tu^2 + t &=& (t^2 u^2 - t)y^2 \neq 0 \\
x^2 - 2tu^2 + t^{-1} &= & t(t^2u^2 -t)z^2 \neq 0
\end{eqnarray*} 
is birational over $\mathbb{Q}(t)$ to an Enriques surface $S$ over $\mathbb{Q}(t)$ with $S(\mathbb{C}\laurent{t})=\emptyset$. 
\end{theorem}

\begin{definition}\label{definition:lafon}
We will refer to the Enriques surface $S$ over $\mathbb{Q}(t)$ in Theorem \ref{thm:lafon} as \emph{Lafon's Enriques surface}.
\end{definition}

The presence of an inf-multiple fibre on a regular model of $S$ over $\mathbb{P}^1_{\mathbb{Q}}$ implies that $S$ does not have a $\mathbb{Q}(t)$-point. In fact, it even implies that $S$ does not have a $\mathbb{C}\laurent{t}$-point (which sometimes is referred to as a ``local point at $t=0$'').

The nonexistence of a $\mathbb{C}\laurent{t}$-point is in fact \emph{equivalent} to an inf-multiple (non-divisible!) fibre of any reduction of $S$ at ``$t = 0$'' (see Lemma \ref{lemma:local_points}). Applying this observation to Lafon's Enriques surface leads to the following result. 

\begin{theorem}\label{thm:main_insight}
Let $X \to \mathbb{P}^1_{\mathbb{Q}}$ be a morphism with $X$ a smooth projective threefold over $\mathbb{Q}$. If the generic fibre $X_{\mathbb{Q}(t)}$ of $X\to \mathbb{P}^1_{\mathbb{Q}}$ is isomorphic to Lafon's Enriques surface $S$ over $\mathbb{Q}(t)$, then the following statements hold.
\begin{enumerate}
\item The scheme-theoretic fibre $X_0$ of $X \to \mathbb{P}^1_{\mathbb{Q}}$ over $0 \in \mathbb{P}^1(\mathbb{Q})$ is an inf-multiple non-divisible fibre of inf-multiplicity two.
\item For every $t_0 \in \mathbb{P}^1(\mathbb{C}) \setminus \{0\}$, the fibre $X_{t_0}$ has a reduced component.
\end{enumerate}
\end{theorem}
\begin{proof}
Note that
\[ x=0, u=0, y=\sqrt{-1}, z= \frac{\sqrt{-1}}{t\sqrt{t}} \]
defines a $\mathbb{C} \laurent{\sqrt{t}}$-point (even a $\mathbb{Q}(\sqrt{-1})(\sqrt{t})$-point) of the smooth surface in $\mathbb{A}^4_{\mathbb{Q}(t)}$ with coordinates $(u,x,y,z)$ given by
\begin{eqnarray*}
x^2 - tu^2 + t &=& (t^2 u^2 - t)y^2 \neq 0 \\
x^2 - 2tu^2 + t^{-1} &= & t(t^2u^2 -t)z^2 \neq 0.
\end{eqnarray*}
Thus,  by the Lang--Nishimura theorem \cite[Theorem~3.6.11]{PoonenRat}, we deduce   that $S(\mathbb{C}\laurent{\sqrt{t}})\neq\emptyset$. Therefore, since $S(\mathbb{C} \laurent{t}) = \emptyset$ by Lafon's theorem (Theorem \ref{thm:lafon}), it follows from Lemma \ref{lemma:local_points} that $X_0$ is an inf-multiple fibre of inf-multiplicity two. To show that the fibre $X_0$ is non-divisible, we can either appeal to Proposition~\ref{prop:enriques_fibration} or note that $S$ has a $\mathbb{C}\laurent{t^{1/3}}$-point. One such point is given\footnote{J.~Nicaise also found in 2009  this cubic point in an e-mail exchange with the fourth-named author.}  by
\[ x = t^{-1/3}, u = t^{-2/3}, y = t^{-2/3} \sqrt{\frac{1-t^{1/3}+t^{5/3}}{1-t^{1/3}}}, z = t^{-4/3} \sqrt{\frac{1+t^{1/3}-2t^{2/3}}{1-t^{1/3}}} = t^{-4/3} \sqrt{1+2t^{1/3}}. \]

Assume $t_0 \in \mathbb{C}\setminus \{0\}$. Then, as $S(\mathbb{C}(\sqrt{t}))$ is non-empty, we directly infer that $S(\mathbb{C} \laurent{t-t_0})$ is non-empty, i.e., $S$ has a local point at $t_0$. Similarly, consider 
\[ x = 1, u = 1, y = \frac{1}{t \sqrt{1-t^{-1}}}, z = \frac{1}{t}\sqrt{\frac{-2+t^{-1}+t^{-2}}{1-t^{-1}}} \]
This defines a local point of $S_{\mathbb{C}(t)}$ at $\infty$, i.e., $S(\mathrm{Frac}(\widehat{\mathcal{O}_{\mathbb{P}^1_{\mathbb{C}},\infty}})) \neq \emptyset$.    It follows from Lemma \ref{lemma:local_points} that $X_{t_0}$ has a reduced component for every $t_0 \in \mathbb{P}^1(\mathbb{C}) \setminus \{0\}$.
\end{proof}
 
\begin{remark} 
We stress that our proof of Theorem~\ref{thm:main_insight} does not need Proposition~\ref{prop:enriques_fibration} (which relies on the main result of \cite{ELW}), as we verified all the necessary properties of Lafon's Enriques surface by direct computation and  Lemma \ref{lemma:local_points}.   
\end{remark}

\begin{remark}
For each odd integer $q\geq 3$, write $s=t^{1/q}$. Then, an explicit $\mathbb{C}\laurent{t^{1/q}}$-point is given by: 
\[ x =s^{-1},~u :=s^{-\frac{q+1}{2}}, ~y :=s^{-\frac{q+1}{2}}\sqrt{\frac{1-s+s^{q+2}}{1-s}}, ~z=s^{-\frac{3q-1}{2}}\sqrt{\frac{1+s^{q-2}-2s^{q-1}}{1-s}}.\]
\end{remark}
 
\begin{remark} 
Let $\mathcal{X}$ be a regular proper model of Lafon's Enriques surface over $\mathbb{P}^1_{\mathbb{Q}}$ such that the reduced subscheme underlying the special fibre over $0$ has simple normal crossings. We can say essentially three things about $\mathcal{X}_0$: 
\begin{enumerate}
\item The fibre $\mathcal{X}_0$ contains components of multiplicities two and three (by Lemma~\ref{lemma:local_points} and Theorem~\ref{thm:main_insight}).
\item There is a closed point $x$ of $\mathcal{X}_0$ such that the gcd of the multiplicities of the irreducible components containing $x$ equals $1$ (Remark \ref{remark:spec_index}).
\item The fundamental group of $\mathcal{X}_0$ is $\mathbb{Z}/2\mathbb{Z}$ (see Section \ref{section:K3}). 
\end{enumerate} 
Note that only (reduced) semi-stable models with torsion canonical bundle (i.e., so-called ``Kulikov models'') have been classified, so that the fibre $\mathcal{X}_0$ lies outside of the presently charted territory. 
\end{remark}

\section{Orbifold Mordell (hence abc) contradicts the Weakly Special Conjecture}\label{section:main_results}
Every smooth projective model of Lafon's Enriques surface over $\mathbb{P}^1_{\mathbb{Q}}$ is special. However, by taking suitable ramified covers of $\mathbb{P}^1$, we get a plethora of counterexamples to the Weakly Special Conjecture, assuming the abc conjecture holds. 

\begin{theorem}\label{thm2}
Let $X \to \mathbb{P}^1_{\mathbb{Q}}$ be a morphism with $X$ a smooth projective threefold such that $X_{\mathbb{Q}(t)}$ is isomorphic to Lafon's Enriques surface. Let $C$ be a smooth projective curve over a field $K$ of characteristic zero. Let $f \colon C \to \mathbb{P}^1_K$ be a finite morphism which is \'etale over $0$. Let $X_f$ be a desingularization of $X \times_{\mathbb{P}^1_{\mathbb{Q}}} C$. Then, the following statements hold.
\begin{enumerate}
\item The threefold $X_f$ is a smooth projective threefold, the generic fibre of $X_f\to C$ is an Enriques surface, the orbifold base of the morphism $X_f\to C$ is $(C, \frac{1}{2} [f^{-1}(0)])$, and $X_f \to C$ has no divisible fibres.
\item The threefold $X_f$ is weakly special if and only if $C$ is of genus at most one.
\item The threefold $X_f$ is special if and only if $C$ is of genus zero and $\deg f \leq 4$.
\end{enumerate}
\end{theorem}
\begin{proof}
The generic fibre of $X_f\to C$ is an Enriques surface, as the generic fibre of $X \to \mathbb{P}^1_{\mathbb{Q}}$ is an Enriques surface. By the characterization of gcd-multiple fibres via local-rational points (Lemma \ref{lemma:index}), the morphism $X_f\to C$ has no divisible fibres, since $X\to \mathbb{P}^1_{\mathbb{Q}}$ has no divisible fibres (Theorem \ref{thm:main_insight}). 

Now, by Theorem \ref{thm:main_insight}~$(i)$, the fibre $X_0$ is an inf-multiple fibre of inf-multiplicity two. Therefore, since $f$ is \'etale over $0$, for every $c$ in $f^{-1}(0)$, the fibre of $X_f\to C$ over $c$ is inf-multiple of inf-multiplicity two as well. On the other hand, suppose that $c$ is a point of $C$ lying over a point $t \in \mathbb{P}^1\setminus\{0\}$. Then, as $X_t$ has a reduced component (Theorem \ref{thm:main_insight}~$(ii)$), the fibration $X \to \mathbb{P}^1$ has a formal local section (Lemma \ref{lemma:local_points}). This implies that $X_f\to C$ has a formal local section around any point $c$ not mapping to $0$. In particular, again by Lemma \ref{lemma:local_points}, the fibre of $X_f \to C$ over a point $c$ not mapping to $0$ has a reduced component.
 This shows that the orbifold base of $X_f\to C$ is $(C,\frac{1}{2}[f^{-1}(0)])$ and proves $(i)$.

If $X_f$ is weakly special, then obviously $C$ has genus at most one. The converse follows from Theorem~\ref{thm:ws}. This proves $(ii)$. 

To prove $(iii)$, assume first that $X_f$ is special. Then, since the fibration $X_f \to C$ defines an orbifold morphism $X_f\to (C,\frac{1}{2} [f^{-1}(0)])$, it follows that the orbifold base $(C, \frac{1}{2} [f^{-1}(0)])$ is not of general type, i.e., the inequality $2g(C)-2+\frac{1}{2}\deg f\leq 0$ holds. Since $\deg(f) \geq 1$, this implies that $g(C)=0$ and $\deg f \leq 4$. Conversely, if $\deg f\leq 4$ and $C$ is of genus zero, then $(C,\frac{1}{2} [f^{-1}(0)])$ is special. Consequently, as the threefold $X_f$ fibres in special surfaces over $C$ with a special orbifold base, we conclude that $X_f$ is special. 
\end{proof}

As already mentioned in the introduction, the following result is Theorem \ref{thm1}. 

\begin{corollary}\label{corollary:first_theorem}
For every integer $d \geq 1$, there exists a smooth projective threefold $X_d$ over $\mathbb{Q}$ and a morphism $X_d \to \mathbb{P}^1_{\mathbb{Q}}$ such that the following properties hold.
\begin{enumerate}
\item The generic fibre of $X_d \to \mathbb{P}^1_{\mathbb{Q}}$ is an Enriques surface over $\mathbb{Q}(t)$.
\item There are exactly $d$ points in $\mathbb{P}^1(\overline{\mathbb{Q}})$ over which the fibre of $f$ is inf-multiple, and all of these fibres have inf-multiplicity two and are non-divisible.
\end{enumerate}
\end{corollary}
\begin{proof}
Let $f \colon \mathbb{P}^1 \to \mathbb{P}^1$ be the morphism given by the polynomial $z^d-1$. Note that $f$ is of degree $d$ and \'etale over $0$. Let $X_d$ be a desingularization of the pullback of $X \to \mathbb{P}^1_{\mathbb{Q}}$ along $f$. Then, by Theorem~\ref{thm2}~$(i)$, the generic fibre of $X_d \to \mathbb{P}^1_{\mathbb{Q}}$ is an Enriques surface, and $X_d \to \mathbb{P}^1_{\mathbb{Q}}$ has precisely $d$ inf-multiple fibres (lying over the $d$-th roots of unity in $\mathbb{P}^1(\overline{\mathbb{Q}})$) and these fibres are of inf-multiplicity two and are not divisible. This concludes the proof. 
\end{proof}

\begin{proof}[Proof of Theorem \ref{thm:abc}] 
Let $X_d$ be as in Theorem \ref{thm1}. Note that $X_d$ is weakly special by Theorem~\ref{thm:ws}. Moreover, since $d \geq 5$, the orbifold base of $X_d \to \mathbb{P}^1$ is of general type. Thus, the Orbifold Mordell Conjecture implies that the orbifold base is Mordellic. In particular, $X_d$ is not arithmetically special, i.e., for every number field $K$, the set $X_d(K)$ is not dense.  
\end{proof}

\begin{remark}
If we take $C$ to be a curve of genus one in Theorem \ref{thm2} then the resulting smooth projective weakly special non-special threefold has infinite fundamental group. The previous three-dimensional examples of Bogomolov--Tschinkel \cite{BT}, Campana--P\u{a}un \cite{CP07} and Rousseau--Turchet--Wang \cite{RTW} are all simply connected.
\end{remark}

\begin{remark}
Let $X \to \mathbb{P}^1_{\mathbb{Q}}$ be a morphism with $X$ a smooth projective threefold over $\mathbb{Q}$ such that the generic fibre $X_{\mathbb{Q}(t)}$ of $X\to \mathbb{P}^1_{\mathbb{Q}}$ is isomorphic to Lafon's Enriques surface $S$ over $\mathbb{Q}(t)$. Then, it follows from Theorem \ref{thm2} that $X$ is special.
\end{remark}

Although Theorem \ref{thm:ws_contradicts_abc} follows from Theorem \ref{thm:abc}, we note that it also  results from the following more precise result.

\begin{corollary} \label{cor:final}
Let $C$ be a smooth projective curve of genus at most one over a number field~$K$. Let $f \colon C \to \mathbb{P}^1_K$ be a finite morphism which is \'etale over $0$ such that $2g(C)-2+\frac{1}{2}\deg f > 0$. Let $X_f$ be a desingularization of $X_K \times_{\mathbb{P}^1_{K}} C$. 
\begin{enumerate}
\item The threefold $X_f$ is weakly special.
\item Assume the Orbifold Mordell Conjecture holds. Then $X_f$ is not arithmetically special.
\item Assume the abc conjecture holds. Then $X_f$ is not arithmetically special.
\end{enumerate}
\end{corollary}
\begin{proof}  Note that $(i)$ follows from Theorem \ref{thm2}.(ii).
Since $2g(C)-2+\frac{1}{2}\deg f>0$, the orbifold $(C,\frac{1}{2}[f^{-1}(0)])$ is of general type. Thus, if Orbifold Mordell holds, then $(C,\frac{1}{2}[f^{-1}(0)])$ is Mordellic. Note that $X_f \to (C, \frac{1}{2}[f^{-1}(0)])$ is a surjective orbifold morphism. Therefore, if $X_f$ were arithmetically special, it would follow from the definitions that $(C,\frac{1}{2}[f^{-1}(0)])$ is arithmetically special, and thus not Mordellic. Contradiction. This proves $(ii)$. Now $(iii)$ follows from $(ii)$, as abc implies Orbifold Mordell (Theorem \ref{thm:abc_implies_orb_mordell}). 
\end{proof}

\begin{remark}\label{remark:strategy}
Let $n \geq 5$ be an integer, let $K$ be a number field and let $a_1,\ldots, a_n\in K$ be pairwise distinct, so that $(\mathbb{P}^1_K, \frac{1}{2}[a_1]+\ldots+\frac{1}{2}[a_n])$ is of general type. Let $f \colon \mathbb{P}^1_K \to \mathbb{P}^1_K$ be the map given by the polynomial $f(x) = \prod_{i=1}^n (x-a_i)$. Then, by Theorem \ref{thm2}, the Weakly Special Conjecture is false, assuming that the Orbifold Mordell Conjecture holds for the (single!)\ orbifold curve $(\mathbb{P}^1_{K}, \frac{1}{2}[a_1]+\ldots+\frac{1}{2}[a_n])$.
\end{remark}

\begin{remark}\label{remark:core}
With notation as in Theorem \ref{thm2}, the core of $X_f$ is trivial if and only if $\deg f \leq 4$ and $C$ has genus zero.
Moreover, if $\deg(f) \geq 5$ or $C$ is of genus at least one, then the core of~$X_f$ is precisely the orbifold base $(C, \frac{1}{2} [f^{-1}(0)])$ of the fibration $X_f\to C$. The main novelty of this article 
is the existence of weakly special threefolds with non-trivial one-dimensional core (namely, most of the $X_f$).
\end{remark}

\section{Disproving complex-analytic analogues of the Weakly Special Conjecture}

In this section we work over $\mathbb{C}$ and disprove (all reasonable) complex-analytic analogues of the Weakly Special Conjecture.
More precisely, the Weakly Special Conjecture has several natural complex-analytic analogues all of which we expect (and sometimes even already know) to be false. For example, the ``Brody'' analogue of the Weakly Special Conjecture would predict that a smooth projective weakly special variety has a dense entire curve; this was disproved in \cite{CP07} by using Bogomolov--Tschinkel's weakly special threefolds \cite{BT} and extending the work of Bogomolov--McQuillan \cite{Bogomolov, McQuillan} on surfaces of general type with $c_1^2>c_2$ to orbifold surfaces (see also \cite{Rou12}). Moreover, the Kobayashi analogue of the Weakly Special Conjecture, which predicts that a smooth projective weakly special variety has a vanishing Kobayashi pseudometric, is disproved in \cite[Theorem~5]{CWKoba} by an example strengthening the conclusion of \cite{CP07}. 
 
In this section, we use Lafon's Enriques surface to give new (and arguably much easier) counterexamples to several complex-analytic analogues of the Weakly Special Conjecture; see Theorem \ref{thm:ws_analytic} for the final result. Before we specialize the base field to the complex numbers, we briefly discuss the notion of perturbation.

\subsection{Perturbing orbifolds of general type} 
Let $\Delta = \sum_i \left(1-\frac{1}{m_i}\right) D_i$ be an orbifold divisor on a normal variety $X$ over a field $k$ of characteristic zero. If $m_i =\infty$, we refer to $D_i$ as a \emph{logarithmic component} of $\Delta$. We may then write $\Delta = \Delta^{\log} + \Delta^{\mathrm{fin}}$, where \[\Delta^{\log} := \sum_{\substack{i \\ m_i = \infty}} D_i = \lfloor \Delta \rfloor\] and $\Delta^{\mathrm{fin}} := \Delta-\Delta^{\log}$. A \emph{finite perturbation} of $(X,\Delta)$ is an orbifold divisor $\Delta^+$ on $X$ with finite multiplicities such that $\supp \Delta^+ \subseteq \supp \Delta$ and, for every divisor $D \subset X$, the multiplicity of $D$ in $\Delta^+$ is at most the multiplicity of $D$ in $\Delta$. In this case, we also refer to the orbifold $(X,\Delta^+)$ as a \emph{finite perturbation} of $(X,\Delta)$.

%\begin{example}\label{example:elliptic_curve}
%Let $E$ be an elliptic curve and let $\Delta = [0]$ be the origin (with multiplicity $\infty$). Note that $\Delta^{\log} = \Delta$ and that $(E,\Delta)$ is of general type. Define $\Delta_m = (1-\frac{1}{m})[0]$. Note that $(E,\Delta_m)$ is a finite perturbation of $(E,\Delta)$. Then, for every $m\geq 2$, the orbifold $(E,\Delta_m)$ is (also) of general type.
%\end{example}
%
%In Example \ref{example:elliptic_curve}, one sees that there exist finite perturbations of (log-general type) punctured elliptic curves which are general type orbifolds. This generalizes as follows. 

\begin{proposition}[Perturbing]\label{prop:perturb_gt}
If $(X,\Delta)$ is a smooth proper orbifold of general type, then there is a finite perturbation $\Delta^+$ of $\Delta$ such that $(X,\Delta^+)$ is of general type.
\end{proposition}
\begin{proof}
Since $K_X+\Delta$ is big, by \cite[Corollary 2.2.24]{Lazzie1}, the divisor $K_X+\Delta-\epsilon \Delta^{\log}$ is big for all sufficiently small $\epsilon > 0$. In particular, there exists an integer $m$ such that $K_X + \Delta - \frac{1}{m} \Delta^{\log}$ is big. Then $\Delta^+ := \Delta - \frac{1}{m} \Delta^{\log}$ is a finite perturbation with the desired properties.
\end{proof} 

Perturbing orbifolds of general type is a standard technique in their study; see for example the proof of \cite[Lemma~5.9]{RTW}.

\begin{remark}[Orbifold curves]\label{remark:perturb_curves}
If $(X,\Delta)$ is a smooth proper orbifold curve of general type, it is possible to write down explicitly a finite perturbation $\Delta^+$ of $\Delta$ such that $(X,\Delta^+)$ is (still) of general type. Indeed, for every component $D$ of $\Delta^{\log}$, define the multiplicity of $D$ in $\Delta^+$ to be $7$. Then, by the classification of orbifold curves of general type, the orbifold $(X,\Delta^+)$ is of general type. 
\end{remark}

We follow \cite[Expos\'e~XII]{SGA1} and let $X^{\an}$ denote the complex-analytic space associated to a finite type scheme $X$ over $\mathbb{C}$.

\subsection{Definitions} 
In this section, let $(X,\Delta)$ be a smooth proper orbifold over $\mathbb{C}$. If $\Delta = \sum_i (1-\frac{1}{m_i})D_i$, we let $\Delta^{\an}$ denote the divisor $\sum_i (1-\frac{1}{m_i}) D_i^{\an}$.
(The pair $(X^{\an},\Delta^{\an})$ is an example of a ``complex-analytic orbifold''; see \cite{Ca11}.)
If $B$ is a complex manifold, a holomorphic map $f \colon B \to X^{\an}$ defines an \emph{orbifold map (with respect to $\Delta^{\an}$)}, denoted by $B\to (X^{\an},\Delta^{\an})$, if $f(B) \not\subset \supp \Delta^{\an}$ and, for every $i$ and every irreducible component $E$ of $f^{-1}(D_i^{\an})$, the multiplicity of $E$ in $f^* D_i^{\an}$ is at least $m_i$. 
 
\begin{definition}
The \emph{Kobayashi pseudometric} $d_{(X,\Delta)}$ on $( X\setminus \Delta^{\log})^{\an}$ is the largest pseudometric $d$ on $(X\setminus \Delta^{\log})^{\an}$ such that, for every holomorphic orbifold map $f\colon \mathbb{D} \to (X^{\an},\Delta^{\an})$ and every $a,b\in\mathbb{D}$, the inequality 
\[ d_{\mathbb{D}}(a,b) \geq d(f(a),f(b)) \]
holds.
\end{definition}

With this notation, we have $d_B = d_{(B,0)}$. Moreover, if $f \colon B\to (X^{\an},\Delta^{\an})$ is an orbifold map and $a,b\in B$, then $d_B(a,b)\geq d_{(X ,\Delta )}(f(a),f(b))$. 

We follow \cite{CWKoba} and study an orbifold extension of Kobayashi hyperbolicity (resp.\ Brody hyperbolicity). We also introduce orbifold analogues of the notions of Picard hyperbolicity and Borel hyperbolicity. (Note that Picard hyperbolicity is the same as the $\Delta^\ast$-extension property in \cite[Definition~3.10]{JKuch}.)

\begin{definition}
Let $Z \subset X$ be a closed subset.
We say that $(X,\Delta)$ is \emph{Kobayashi hyperbolic modulo $Z$} if, for all $p,q \in X^{\an}$ with either $p \notin Z^{\an}$ or $q \notin Z^{\an}$, the inequality $d_{(X,\Delta)}(p,q)>0$ holds. 
We say that $(X,\Delta)$ is \emph{Picard hyperbolic modulo $Z$} if, for every holomorphic orbifold map $f\colon \mathbb{D}^\ast \to (X,\Delta)$ with $f(\mathbb{D}^\ast) \not\subset Z^{\an}$, the holomorphic map $\mathbb{D}^\ast\to X^{\an}$ extends to a holomorphic map $\mathbb{D}\to X^{\an}$.
We say that $(X,\Delta)$ is \emph{Borel hyperbolic modulo $Z$} if, for every (reduced) variety $S$ over $\mathbb{C}$, every holomorphic orbifold map $f\colon S^{\an}\to (X^{\an},\Delta^{\an})$ with $f(S^{\an})\not\subset Z^{\an}$ is algebraic.
We say that $(X,\Delta)$ is \emph{Brody hyperbolic modulo $Z$} if every orbifold map $f\colon \mathbb{C}\to (X^{\an},\Delta^{\an})$ with $f(\mathbb{C})\not\subset Z^{\an}$ is constant. 
\end{definition}

\begin{definition}
We say that $(X,\Delta)$ is \emph{Kobayashi pseudo-hyperbolic} if there exists a proper (Zariski\nobreakdash-)closed subset $Z \subset X$ such that $(X,\Delta)$ is Kobayashi hyperbolic modulo $Z$. We define \emph{Picard pseudo-hyperbolicity}, \emph{Borel pseudo-hyperbolicity} and \emph{Brody pseudo-hyperbolicity} analogously. 
\end{definition}
 
\begin{definition}
We say that $(X,\Delta)$ is \emph{Kobayashi hyperbolic} if it is Kobayashi hyperbolic modulo the empty subset $Z:=\emptyset \subset X$. We define \emph{Picard}, \emph{Borel} and \emph{Brody hyperbolicity} of $(X,\Delta)$ analogously.
\end{definition}

\subsection{From Kobayashi to Picard to Borel to Brody}
Let us start with the fact that Kobayashi hyperbolic orbifolds (with no logarithmic components) are Picard hyperbolic; this is implicit in Kwack's work on Great Picard Theorem for compact hyperbolic varieties \cite{Kwack-2}.

\begin{theorem}[Kwack]\label{thm:kwack}
Let $(X,\Delta)$ be a smooth proper orbifold over $\mathbb{C}$. Assume that the multiplicities of $\Delta$ are finite. If $(X,\Delta)$ is Kobayashi hyperbolic, then $(X,\Delta)$ is Picard hyperbolic.
\end{theorem}
\begin{proof}
Since the multiplicities of $\Delta$ are finite (i.e., $\Delta^{\log} = 0$) and $(X,\Delta)$ is Kobayashi hyperbolic, the pseudometric $d_{(X,\Delta)}$ is a metric on $X^{\an}$. Let $f \colon \mathbb{D}^\ast \to (X^{\an},\Delta^{\an})$ be a holomorphic map. Note that condition $(a)$ in Kwack's theorem \cite[Theorem~6.3.4]{KobayashiBook} holds, as $d_{\mathbb{D}^\ast} \geq f^\ast d_{(X,\Delta)}$. Moreover, the second condition $(b)$ in \emph{loc.\ cit.}\ holds by compactness of $X^{\an}$. Thus, \emph{loc.\ cit.}\ implies that $f$ extends to a holomorphic map $\mathbb{D} \to X^{\an}$.
\end{proof}

We do not know whether a Kobayashi pseudo-hyperbolic orbifold is Picard pseudo-hyperbolic (even for the trivial orbifold structure $\Delta=0$ and $X$ smooth projective). 
 
\begin{lemma}\label{lem:borel_brody}
Let $(X,\Delta)$ be an orbifold and let $Z \subset X$ be a closed subset. 
\begin{enumerate}
\item If $(X,\Delta)$ is Picard hyperbolic modulo $Z$, then $(X,\Delta)$ is Borel hyperbolic modulo $Z$.
\item If $(X,\Delta)$ is Borel hyperbolic modulo $Z$, then $(X,\Delta)$ is Brody hyperbolic modulo $Z$.
\end{enumerate}
\end{lemma}
\begin{proof} 
To prove $(i)$, we adapt the proof of \cite[Corollary~3.11]{JKuch}. Namely, let $S$ be a variety and let $f \colon S^{\an} \to (X^{\an},\Delta^{\an})$ be a holomorphic orbifold map. To prove that $f \colon S^{\an}\to X^{\an}$ algebraizes, we may replace $S$ by a dense open. Thus, we may assume that $S$ is a smooth affine variety. We now argue by induction on $\dim S$. 

Assume that $\dim S = 1$. Let $\overline{S}$ be a smooth projective curve with function field $\mathbb{C}(S)$. To prove algebraicity of $f$, it suffices to show that $f$ extends to a holomorphic map $\overline{S}^{\an}\to X^{\an}$ as every such map is algebraic \cite[Corollaire~XII.4.5]{SGA1}. To show that $f$ extends, let $p \in \overline{S}\setminus S$ and choose a small open disk $\mathbb{D}\subset \overline{S}^{\an}$ centered at $p$ such that $\mathbb{D}\cap \overline{S}\setminus S = \{p\}$ (i.e., $\mathbb{D}$ only contains the ``puncture'' $p$ from the boundary of $S$). Then, the holomorphic orbifold map $f\colon S^{\an}\to (X^{\an},\Delta^{\an})$ restricts to a holomorphic map $f|_{\mathbb{D}^\ast} \colon \mathbb{D}^\ast \to X^{\an}$. Since $f(\mathbb{D}^\ast)\not\subset Z^{\an}\cup \supp \Delta^{\an}$, the map $f|_{\mathbb{D}^\ast}$ defines a holomorphic orbifold map $f|_{\mathbb{D}^\ast}\colon \mathbb{D}^\ast \to (X^{\an}, \Delta^{\an})$, so that our assumption on Picard hyperbolicity modulo $Z$ implies that $f|_{\mathbb{D}^\ast}\colon \mathbb{D}^\ast \to X^{\an}$ extends to a holomorphic map $\mathbb{D}\to X^{\an}$. This implies that $f$ extends to a holomorphic map $\overline{S}^{\an}\to X^{\an}$ as required. 

If $\dim S > 1$, the proof of \cite[Proposition~3.6]{JKuch} shows that there is a smooth hypersurface $H \subset S$ such that $S^{\an}\to X^{\an}$ is algebraic if and only if $H^{\an}\to S^{\an}$ is algebraic and $f(H^{\an})\not\subset Z^{\an}\cup \supp \Delta^{\an}$. (In fact, the very general $H\subset S$ has the desired property.) Since the image of the composed map $H^{\an}\subset S^{\an}\to X^{\an}$ does not lie in $\supp \Delta^{\an}$, the composition $H^{\an}\subset S^{\an}\to X^{\an}$ is an orbifold holomorphic map $H^{\an}\to (X^{\an},\Delta^{\an})$. Therefore, by the induction hypothesis, it is algebraic. This implies that $S^{\an}\to X^{\an}$ is algebraic, as desired. 

To prove $(ii)$, we adapt the proof of \cite[Lemma~3.2]{JKuch} and suppose that $\varphi \colon \mathbb{C} \to (X^{\an},\Delta^{\an})$ is a non-constant holomorphic orbifold map. If $\varphi$ is not algebraic, then we are done. If $\varphi$ is algebraic, then $\varphi \circ \exp \colon \mathbb{C} \to \mathbb{C}\to X^{\an}$ is not algebraic. Since $\exp$ is dominant, the composition $\varphi\circ \exp$ is a non-algebraic holomorphic orbifold map $\mathbb{C}\to (X^{\an},\Delta^{\an})$. This shows that $(X^{\an},\Delta^{\an})$ is not Borel hyperbolic.
\end{proof}

\subsection{Orbifold curves}
In the case of orbifold curves, all notions of hyperbolicity coincide with the property of being of general type.
 
\begin{theorem}\label{thm:orbifold_curves} 
Let $(X,\Delta)$ be a smooth proper orbifold curve. Then the following are equivalent. 
\begin{enumerate}
\item The orbifold $(X,\Delta)$ is of general type (i.e., $\deg K_X +\deg \Delta >0$). 
\item The orbifold $(X,\Delta)$ is Kobayashi hyperbolic. 
\item The orbifold $(X,\Delta)$ is Picard hyperbolic. 
\item The orbifold $(X,\Delta)$ is Borel hyperbolic. 
\item The orbifold $(X,\Delta)$ is Brody hyperbolic. 
\end{enumerate}
\end{theorem} 
\begin{proof}
First assume that $\Delta$ has finite multiplicities. Then, the implication $(i) \implies (ii)$ follows from \cite[\S 12, Theorem 4]{CWKoba}. Since the multiplicities of $\Delta$ are finite, $(ii) \implies (iii)$ follows from Theorem \ref{thm:kwack}. The implications $(iii) \implies (iv)$ and $(iv) \implies (v)$ follow from Lemma \ref{lem:borel_brody}. Finally, if $(X,\Delta)$ is not of general type, then there is an elliptic curve $E$ and a non-constant morphism $E\to (X,\Delta)$, so that $(X,\Delta)$ is not Brody hyperbolic. This proves that $(v) \implies (i)$ and concludes the proof for all orbifold curves $(X,\Delta)$ with $\Delta^{\log} = 0$.
 
Now, let us treat the general case (in which some of the components of $\Delta$ might have multiplicity $\infty$). If $(X,\Delta)$ is not of general type, then either there is a non-constant orbifold morphism $\mathbb{A}^1\setminus \{0\} \to (X,\Delta)$ or an elliptic curve $E$ and a non-constant orbifold morphism $E \to (X,\Delta)$. This immediately implies that $(X,\Delta)$ is not Kobayashi hyperbolic, nor Picard hyperbolic, nor Borel hyperbolic, nor Brody hyperbolic. Thus, $(ii) \implies (i)$, $(iii) \implies (i)$, $(iv) \implies (i)$ and $(v) \implies (i)$. 

Now, to conclude the proof, assume that $(i)$ holds, i.e., $(X,\Delta)$ is of general type. Choose a finite perturbation $(X,\Delta^+)$  that is (still) of general type; see Proposition \ref{prop:perturb_gt}. (Instead of Proposition \ref{prop:perturb_gt}, we could also appeal to the classification of orbifold curves of general type; see Remark \ref{remark:perturb_curves}.) Now, $(X,\Delta^+)$ is Kobayashi hyperbolic, Picard hyperbolic, Borel hyperbolic and Brody hyperbolic by the first paragraph of this proof. Since the identity map $(X,\Delta) \to (X,\Delta^+)$ is an orbifold map, it follows that $(X,\Delta)$ is Kobayashi hyperbolic, Picard hyperbolic, Borel hyperbolic and Brody hyperbolic. We have shown that $(i) \implies (ii)$, $(i) \implies (iii)$, $(i) \implies (iv)$ and $(i) \implies (v)$. This concludes the proof.
\end{proof}

\subsection{Analytic notions of specialness}
The opposite notion of Kobayashi pseudo-hyperbolicity would be that the Kobayashi pseudometric vanishes everywhere. We follow \cite{Ca11} and formalize the ``opposite'' behaviour in the Kobayashi, Picard, Borel, and Brody setting as follows. 
 
\begin{definition}
Let $(X,\Delta)$ be a smooth proper orbifold pair. We say that $(X,\Delta)$ is \emph{Kobayashi-special} if $d_{(X,\Delta)}$ is identically zero. We say that $(X,\Delta)$ is \emph{Picard-special} if there is a holomorphic orbifold map $f\colon \mathbb{D}^\ast\to (X^{\an},\Delta^{\an})$ such that the closure $\overline{\Gamma_f}$ in $\mathbb{D}\times X^{\an}$ of its graph $\Gamma_f \subset \mathbb{D}^\times \times X^{\an}$ satisfies $\overline{\Gamma_f}\cap (\{0\}\times X^{\an}) = X$. We say that $(X,\Delta)$ is \emph{Borel-special} if there is a smooth affine curve $C$ and a holomorphic orbifold map $C^{\an} \to (X^{\an},\Delta^{\an})$ whose graph in $C^{\an} \times X^{\an}$ is Zariski-dense. We say that $(X,\Delta)$ is \emph{Brody-special} if there is a holomorphic orbifold map $\mathbb{C}\to (X^{\an},\Delta^{\an})$ whose image in $X^{\an}$ is Zariski dense.
\end{definition}

If $(X,\Delta)$ is Brody-special, then $(X,\Delta)$ is easily seen to be Borel-special. It is not clear whether such an $(X,\Delta)$ is Picard-special, nor Kobayashi-special. However, it is worth noting that if in addition $(X,\Delta)$ admits a holomorphic orbifold map $\mathbb{C}\to (X^{\an},\Delta^{\an})$ whose image is metrically dense, then $(X,\Delta)$ is Kobayashi-special and Picard-special.

The many conjectures on special varieties in \cite{Ca11} naturally extend to the following.
\begin{conjecture}
If $(X,\Delta)$ is a smooth proper orbifold pair, the following are equivalent. 
\begin{enumerate}
\item The orbifold $(X,\Delta)$ is special (see \cite[Definition~8.1]{Ca11}).
\item The orbifold $(X,\Delta)$ is Kobayashi-special.
\item The orbifold $(X,\Delta)$ is Picard-special.
\item The orbifold $(X,\Delta)$ is Borel-special.
\item The orbifold $(X,\Delta)$ is Brody-special.
\end{enumerate}
\end{conjecture}

By Theorem \ref{thm:orbifold_curves}, this conjecture holds for smooth proper orbifold curves. 

\begin{theorem} \label{thm:ws_analytic}
There exists a simply connected smooth projective threefold $X$ over $\mathbb{C}$ such that the following properties hold.
\begin{enumerate}
\item $X$ is weakly special, but not special.
\item $X$ is not Kobayashi-special.
\item $X$ is not Picard-special
\item $X$ is not Borel-special.
\item $X$ is not Brody-special.
\end{enumerate}
\end{theorem}
\begin{proof} 
Let $X = X_d$ be as in Theorem \ref{thm1}. Then there is a family of Enriques surfaces $X \to \mathbb{P}^1$ with orbifold base $(\mathbb{P}^1,\Delta)$ of general type, so that $X$ is not special. By Corollary \ref{cor:enriques_fibration}, the threefold $X$ is weakly special. Now, since the orbifold base is an orbifold curve of general type, it follows that it is Kobayashi hyperbolic, Picard hyperbolic, Borel hyperbolic and Brody hyperbolic by Theorem \ref{thm:orbifold_curves}. Since $X\to (\mathbb{P}^1,\Delta)$ is flat surjective, it is an orbifold morphism. Thus, $X$ is not Kobayashi-special, Picard-special, Borel-special, nor Brody-special. (Indeed, since $X\to (\mathbb{P}^1,\Delta)$ is surjective, if $X$ were Kobayashi-special, the same would hold for $(\mathbb{P}^1,\Delta)$. Contradiction. The argument for Picard, Borel, and Brody is similar.)
\end{proof}

\section{Non-isotriviality of Lafon's Enriques surface}\label{section:fundamental_group}

In this section we prove that Lafon's Enriques surface is non-isotrivial. Recall that if $k$ is a field and $k \subseteq K$ is a field extension, a variety $V$ over $K$ is \emph{isotrivial (with respect to $k \subseteq K$)} if $V_{\overline{K}}$ can be defined over $\overline{k}$.

\begin{theorem}\label{thm:lafon_is_non_isotrivial}
If $S$ is Lafon's Enriques surface over $\mathbb{Q}(t)$, then $S$ is non-isotrivial. 
\end{theorem}   
 
To explain our (slightly surprising) strategy for proving Theorem \ref{thm:lafon_is_non_isotrivial}, let us momentarily consider any Enriques surface~$S$ over $\mathbb{C}(t)$ with a regular projective model $\mathcal{X}$ over $\mathbb{P}^1_{\mathbb{C}}$. If $S$ is isotrivial, then $S$ obviously has ``potentially good reduction'', i.e., for some curve $C$ with function field $K$ and some finite morphism $C \to \mathbb{P}^1_{\mathbb{C}}$, the surface $S_K$ has a smooth proper model $\mathcal{X}\to C$ over $C$. The existence of such a smooth proper model implies the liftability of local rational points along a K3 double cover (Lemma \ref{lemma:nonliftability_extendability}). More precisely, let $\Omega$ be a finite extension of a completion of $K$ (e.g., $\Omega = \mathbb{C}\laurent{t^{1/n}}$ for some $n\geq 1$). Then the Enriques surface $S_\Omega$ has precisely two K3 double covers over $\Omega$. (In the case of Lafon's Enriques surface, we can write down these K3 covers explicitly; see the proof of Lemma \ref{lemma:k3cover_doesnt_extend}.) Now, by the aforementioned Lemma \ref{lemma:nonliftability_extendability}, any $\Omega$-point of $S$ lifts along one of these K3 double covers. If $\alpha$ is the golden ratio and $\Omega$ is $\mathbb{C}\laurent{(t+\alpha)^{1/n}}$ for any $n \geq 1$, we can verify that, for Lafon's Enriques surface, there are $\Omega$-points on $S$ which do not lift to an $\Omega$-point on either of the two K3 double covers of $S_\Omega$. This means that $S$ does not have potentially good reduction and thus is non-isotrivial. We emphasize that it is the apparently ``very singular'', but not inf-multiple, fibre over $t=-\alpha$ which plays an important role here, rather than the inf-multiple fibre at $t=0$.

We start this section with some explicit computations related to liftability of local points along K3 double covers of Lafon's Enriques surface.

\begin{lemma} \label{lemma:nonliftability1}
Let $n$ be a positive integer and $\alpha \in \mathbb{C}$ such that $\alpha^2 = \alpha+1$. Then, there exist $u,x,y,z$ in $\mathbb{C}\laurent{t^{1/(2n)}}$ such that 
\begin{equation} \label{equation:shifted_lafon} \tag{$*$}
\begin{aligned}
x^2 - tu^2 + \alpha u^2 + t-\alpha & = ( (t-\alpha)^2 u^2 -t+\alpha) y^2 \neq 0\\ 
x^2 - 2tu^2 + 2\alpha u^2 +\frac{1}{t-\alpha} &= (t-\alpha)( (t-\alpha)^2u^2 - t +\alpha) z^2 \neq 0
\end{aligned}
\end{equation}
and such that neither
\[  (t-\alpha)^2 u^2 - t + \alpha \] 
nor
\[ t^{1/n}( (t-\alpha)^2u^2 - t + \alpha) \]
is a square in $\mathbb{C}\laurent{t^{1/(2n)}}$.
\end{lemma}
\begin{proof}
Take $u = \frac{1}{\sqrt{-\alpha}} + t^{1/(2n)}$ and $x=\alpha$. Then $(t-\alpha)^2u^2 - t + \alpha$ is not a square in $\mathbb{C}\laurent{t^{1/(2n)}}$, as it has valuation $\frac{1}{2n}$. Since $t^{1/n}$ is a square in $\mathbb{C}\laurent{t^{1/(2n)}}$, it follows that $t^{1/n}((t-\alpha)^2u^2 - t+\alpha)$ is also not a square in $\mathbb{C}\laurent{t^{1/(2n)}}$. Now, to conclude the proof, it suffices to show that there exist $y$ and $z$ in $\mathbb{C}\laurent{t^{1/(2n)}}$ such that $u,x,y,z$ satisfy Equations (\ref{equation:shifted_lafon}). To see this, it suffices to prove that the left-hand sides of the equations are not squares in $\mathbb{C}\laurent{t^{1/(2n)}}$. This can be done by calculating the valuations of the left-hand sides and noticing that both of them are equal to~$\frac{1}{2n}$. 
% Indeed, the constant term of $x^2 - tu^2 +\alpha u^2 + t-\alpha$ equals $0$ and the next significant term is $2\alpha \frac{1}{\sqrt{-\alpha}} t^{1/(2n)}$ (hence has valuation $\frac{1}{2n}$. Similarly, the constant term of $x^2 - 2tu^2 + 2\alpha u^2 +\frac{1}{t-\alpha}$ vanishes and the next significant term is $4\alpha \frac{1}{\sqrt{-\alpha}}t^{1/(2n)}$ (hence has valuation $1/(2n)$).
\end{proof}

Shifting the equations above by $\alpha$ immediately gives the following result about the equations defining Lafon's Enriques surface.

\begin{lemma} \label{lemma:nonliftability2}
Let $n$ be a positive integer and $\alpha \in \mathbb{C}$ such that $\alpha^2 = \alpha+1$. Then, there exist $u,x,y,z$ in $\mathbb{C}\laurent{(t+\alpha)^{1/(2n)}}$ such that 
\begin{equation*}
x^2 - tu^2 + t = ( t^2u^2 -t) y^2 \neq 0 \quad \text{and} \quad x^2 - 2tu^2 + t^{-1} = t( t^2u^2 - t) z^2 \neq 0
\end{equation*}
and such that neither $ t^2u^2 - t$ nor $(t+\alpha)^{1/n}( t^2u^2 - t)$ is a square in $\mathbb{C}\laurent{(t+\alpha)^{1/(2n)}}$.
\end{lemma}
\begin{proof}
After performing the change of variables $t \mapsto t-\alpha$, this becomes Lemma~\ref{lemma:nonliftability1}.
\end{proof}

To show that a given K3 double cover of Lafon's Enriques surface $S$ over $\mathbb{Q}(t)$ does not extend to an \'etale cover of any smooth projective model of $S$ over $\mathbb{P}^1$, we will use that non-liftability of ``local points'' on the generic fibre is an obstruction.

\begin{lemma}\label{lemma:nonliftability_extendability}
Let $B$ be an integral normal noetherian one-dimensional scheme with function field $K$ such that $B$ has no nontrivial finite \'etale covers. Let $X \to Y$ be a surjective finite \'etale morphism of proper $K$-schemes. If there is a $K$-point $p \in Y(K)$ such that $X_p(K) = \emptyset$, then $X \to Y$ does not extend to a finite \'etale morphism $\mathcal{X}\to \mathcal{Y}$ of proper $B$-schemes.
\end{lemma}
\begin{proof}
Suppose that $\mathcal{X}\to \mathcal{Y}$ is a finite \'etale morphism of proper $B$-schemes extending the $K$-morphism $X \to Y$. By the valuative criterion of properness, we may consider $p \in Y(K)$ as a section $\epsilon_p \colon B \to \mathcal{Y}$ of $\mathcal{Y}\to B$. Since $B$ has no nontrivial finite \'etale covers, the finite \'etale morphism $B \times_{\epsilon_p, \mathcal{Y}}\mathcal{X} \to B$ has a section, say $B \to \mathcal{X}$. In particular, we see that $X_p(K)$ is non-empty. 
\end{proof}

\begin{lemma}\label{lemma:k3cover_doesnt_extend}
Let $S$ be Lafon's Enriques surface over $\mathbb{Q}(t)$. Let $\alpha$ be a complex number such that $\alpha^2 = \alpha+1$. Let $n \geq 1$ be an integer and let $\mathcal{X}$ be a proper model of $S_{\mathbb{C}\laurent{(t+\alpha)^{1/n}}}$ over $\mathbb{C}\powerseries{(t+\alpha)^{1/n}}$. If $\mathcal{Y} \to \mathcal{X}$ is a finite \'etale cover, then $\mathcal{Y}_{\overline{\mathbb{C}\laurent{t+\alpha}}}$ is not the K3 double cover of $S_{\overline{\mathbb{C}\laurent{t+\alpha}}}$. 
\end{lemma}
\begin{proof}
Over the affine part of the Enriques surface $S_{\overline{\mathbb{C}\laurent{t+\alpha}}}$ which is described by the equation in Theorem \ref{thm:lafon}, the K3 double cover is described by the following equations in the variables $(u,w,x,y,z)$ (see \cite[Proof of Prop~1.1]{Lafon}):
\begin{eqnarray*}
x^2 - tu^2 + t &=& (t^2 u^2 - t)y^2 \neq 0 \\
x^2 - 2tu^2 + t^{-1} &= & t(t^2u^2 -t)z^2 \neq 0 \\
w^2 &=& t^2 u^2 - t
\end{eqnarray*}
This K3 double cover has, up to isomorphism, two forms over the non-algebraically closed field $\mathbb{C}\laurent{(t+\alpha)^{1/n}}$. The first form is given by the equations above and 
the second form is obtained by replacing the third equation above with the equation
\begin{equation*}
w^2 = (t+\alpha)^{1/n} (t^2 u^2 - t).
\end{equation*}
Now observe that by Lemma \ref{lemma:nonliftability2}, there are, for both $\mathbb{C}\laurent{(t+\alpha)^{1/n}}$-forms of the K3 double cover, $\mathbb{C}\laurent{(t+\alpha)^{1/2n}}$\nobreakdash-rational points of $S_{\mathbb{C}\laurent{(t+\alpha)^{1/n}}}$ which cannot be lifted to that form of the K3 double cover. By Lemma \ref{lemma:nonliftability_extendability}, this implies that both forms of the K3 double cover do not extend to an \'etale double cover defined over $\mathbb{C}\powerseries{(t+\alpha)^{1/n}}$. This finishes the proof.
\end{proof}

\begin{proof}[Proof of Theorem \ref{thm:lafon_is_non_isotrivial}] 
We argue by contradiction. Suppose that $S$ is isotrivial. Let $\alpha$ be a complex number such that $\alpha^2 = \alpha+1$. Let $K=\mathbb{C}\laurent{t+\alpha}$ be the completion of $\mathbb{C}(t)$ with respect to the valuation associated to $t+\alpha$. For every integer $n\geq 1$, let $K_n = \mathbb{C}\laurent{(t+\alpha)^{1/n}}$. Note that $K_n$ is the unique field extension of $K$ of degree $n$. Now, since $S$ is isotrivial, there is an integer $n\geq 1$ such that $S_{K_n}$ can be defined over $\mathbb{C}$, i.e., there is an Enriques surface $S_0$ over $\mathbb{C}$ and an isomorphism $S_{K_n} \cong S_0\otimes_{\mathbb{C}} K_n$. Let $\mathcal{X} = S_0 \times \Spec \mathbb{C}\powerseries{(t+\alpha)^{1/n}}$ and note that $\mathcal{X}$ is a smooth proper model of $S_{K_n}$ over $\mathbb{C}\powerseries{(t+\alpha)^{1/n}}$. Note that $S_0$ has a K3 double cover, say $S'\to S_0$. Let $\mathcal{Y} = S' \times \mathbb{C}\powerseries{(t+\alpha)^{1/n}}$. Then $\mathcal{Y} \to \mathcal{X}$ is a finite \'etale morphism and $\mathcal{Y}_{\overline{\mathbb{C}\laurent{t+\alpha}}}$ is the K3 double cover of $S_{\overline{\mathbb{C}\laurent{t+\alpha}}}$. This contradicts Lemma \ref{lemma:k3cover_doesnt_extend}.
\end{proof}

\begin{remark}\label{remark:lafon_non_isotrivial}
Our proof of non-isotriviality of Lafon's Enriques surface exploits an interesting connection between potential good reduction and isotriviality. In fact, our proof of Theorem \ref{thm:lafon_is_non_isotrivial} can be reformulated as follows. If Lafon's Enriques surface $S$ over $\mathbb{Q}(t)$ were to be isotrivial, then it would have potentially good reduction. In particular, every K3 double cover of $S$ would extend over such a good model. By Lemma \ref{lemma:nonliftability_extendability}, rational points (valued in $\mathbb{C}\laurent{(t+\alpha)^{1/n}}$) would lift on the generic fibre to this K3 double cover. But this is not the case when $\alpha$ is the golden ratio by Lemma \ref{lemma:nonliftability2}.
\end{remark}

\section{Fundamental groups}

In this section we determine the fundamental group of any regular projective model $\mathcal{X}$ of Lafon's Enriques surface $S$. We will use a well-known exact sequence of fundamental groups associated to a fibration, analogous to the homotopy exact sequence for fibrations in algebraic topology. Since we will use this result also in the purely complex-analytic setting, we give the definition of a fibration with no divisible fibres in codimension one in the complex-analytic setting.

\begin{definition}
Let $f \colon X \to Y$ be a proper surjective morphism of complex manifolds with connected fibres. We say that $f$ has \emph{no divisible fibres in codimension one} if, for every prime divisor $D\subset Y$, there is no integer $m>1$ such that, for every prime divisor $E$ of $X$ surjecting onto $D$, the multiplicity of $E$ in $f^*(D)$ is divisible by $m$. 
\end{definition}

The following result is well-known; due to lack of reference we include a self-contained proof.  We stress that some of the arguments in our proof below already appear in the literature (see for example the proof of \cite[Proposition~12.9]{Ca11}). 

\begin{lemma}\label{lemma:nori}  
Let $f \colon X \to Y$ be a proper surjective morphism of complex manifolds with connected fibres. Suppose that $f$ does not have any divisible fibres in codimension one. Then, for every point $y$ in $Y$, the sequence of topological fundamental groups
\[ \pi_1(X_y)\to \pi_1(X)\to \pi_1(Y)\to 1 \]
is exact.
\end{lemma}
\begin{proof}
(The following argument is inspired by the proof of Nori's \cite[Lemma~1.5]{Nori83}, where the stronger assumption that $X \to Y$ has no nowhere reduced fibres in codimension one is employed.) We first prove the lemma when $X_y$ is smooth.    In this case, the result  can also be deduced from \cite[Proposition~12.9]{Ca11} (see Remark \ref{remark:12.9}).   For the reader's comfort, we provide a self-contained proof.

Let $y \in Y$ be a point such that the fibre $X_y$ is smooth.  Let $Y^o\subset Y$ be the maximal open subset of $Y$ such that $f$ is smooth over $Y^o$. Let $X^o = f^{-1}(Y^o)$. Since the induced proper surjective morphism $f|_{X^o} \colon X^o \to Y^o$ is smooth, it follows from Ehresmann's theorem that $f|_{X^o}$ is a locally trivial fibration, so that the sequence 
\[ \pi_1(X_y)\to \pi_1(X^o) \to \pi_1(Y^o) \to 1 \]
is exact. Let $K_1$ (resp.\ $K_2$) be the kernel of $\pi_1(X^o)\to \pi_1(X)$ (resp.\ $\pi_1(Y^o)\to \pi_1(Y)$). Note that the homomorphism $\pi_1(X^o) \to \pi_1(Y^o)$ restricts to a morphism $K_1 \to K_2$. The following diagram commutes:

\begin{equation*} \begin{tikzcd}
 & K_1 \ar[r] \ar[d] & K_2 \ar[d] & \\
\pi_1(X_y) \ar[equal, d] \ar[r] & \pi_1(X^o) \ar[r] \ar[d] & \pi_1(Y^o) \ar[r] \ar[d] & 1  \\
\pi_1(X_y) \ar[r] & \pi_1(X) \ar[r] & \pi_1(Y) \ar[r] & 1
\end{tikzcd} \end{equation*}

We claim that $K_1 \to K_2$ is surjective. Indeed, first note that $K_2$ is generated by loops around the finitely many irreducible closed subsets of $Y \setminus Y^o$ of codimension one (in $Y$). Thus, to show the surjectivity of $K_1 \to K_2$, it suffices to show that for every prime divisor $D \subset Y$ contained in $Y \setminus Y^o$, the class of every simple loop $\gamma$ in $Y^o$ around $D$ can be lifted to an element of $K_1$, i.e., can be lifted to the class of a loop in $X^o$ contractible in $X$. To prove this, we will use that $f$ has no divisible fibres in codimension one. Write $f^{-1}(D) = \sum_{i} a_i F_i + R$, where the $F_i$ surject onto $D$ and the components of $R$ get contracted under $f$. Since $\gcd(a_i) = 1$, there are integers $\lambda_i$ such that $\sum_i \lambda_i a_i = 1$. Now, let $\gamma_i$ be a simple loop around $F_i$ in $X^o$. Note that $\gamma_i$ is contractible in $X$ and that the image of $\gamma_i$ under $f$ is a loop winding $a_i$ times around $D$, so that the class $[\gamma_i]$ maps to $a_i[\gamma]$. Thus, the element $\prod_i [\gamma_i]^{\lambda_i}$ of $\pi_1(X^o)$ maps to $[\gamma]$ in $\pi_1(X)$. Since $\prod_i [\gamma_i]^{\lambda_i}$ is an element of $K_1$ mapping to $[\gamma]$, this finishes the proof of the claim.

Since $K_1\to K_2$ is surjective, $\pi_1(X^o)\to \pi_1(X)$ is surjective and $\pi_1(Y^o)\to \pi_1(Y)$ is surjective, a simple diagram chase shows that the exactness of the middle sequence implies the exactness of the bottom sequence. This proves the lemma in the case that $X_y$ is smooth.

To reduce the general case to the smooth case, let $\mathcal{T} \subset X$ be an open neighbourhood of $X_y$ such that $X_y \subset \mathcal{T}$ is a deformation retract (this exists by \cite[Theorem I.8.8]{BHPV}). We have to show that any element $\alpha$ of the kernel of $\pi_1(X) \to \pi_1(Y)$ is in the image of $\pi_1(X_y) \to \pi_1(X)$. To do so, observe that the properness of the map $X \to Y$ implies that $\mathcal{T}$ contains the preimage of an open neighbourhood of $y$ in $Y$. Consequently, $\mathcal{T}$ contains a smooth fibre $F$ of $X \to Y$. From the first part of this proof, we know that $\alpha$ lies in the image of $\pi_1(F) \to \pi_1(X)$. Since the map $\pi_1(F) \to \pi_1(X)$ factors as $\pi_1(F) \to \pi_1(\mathcal{T}) \to \pi_1(X)$, we see that $\alpha$ lies in the image of $\pi_1(\mathcal{T}) \to \pi_1(X)$. Since $\pi_1(X_y) = \pi_1(\mathcal{T})$, we see that $\alpha$ lies in the image of $\pi_1(X_y) \to \pi_1(X)$, as desired.
\end{proof}
 
\begin{remark}\label{remark:12.9}
If $X_y$ is smooth, the preceding Lemma \ref{lemma:nori} can also be deduced from \cite[Proposition~12.9]{Ca11} as follows. (Note that in \emph{loc.\ cit.}\ it is implicitly assumed that $\Delta_Y$ is the orbifold base of $f:(X,\Delta_X)\to Y$; see middle of \cite[p.~910]{Ca11}.)  Let $Y^*$ be the Zariski open subset of $Y$ over which $f$ has equidimensional fibres.  Define $X^\ast = f^{-1}(Y^\ast)$. Then, the fibration $f^*\colon X^* \to Y^*$ is   neat (see   \cite[Definition~4.8]{Ca11}). Since $f$ does not have any divisible fibres in codimension one,  it follows from \cite[Proposition~12.9]{Ca11} that the sequence   $\pi_1(X_y)\to \pi_1(X^*)\to \pi_1(Y^*)\to 1$ is exact for any general point $y$ of $Y$.  Since  the complement of $Y^*$ in $Y$ has codimension at least $2$, we have  that $\pi_1(Y^*)=\pi_1(Y)$, so that a straightforward diagram chase then shows that the sequence $\pi_1(X_y)\to \pi_1(X)\to \pi_1(Y)\to 1$ is also exact for any general point $y$ in $Y$.
\end{remark}

We are now ready to state and prove the main result of this section.

\begin{theorem}\label{thm:lafon_pi1}
Let $X \to \mathbb{P}^1_{\mathbb{Q}}$ be a morphism with $X$ a smooth projective threefold such that the generic fibre $X_\eta$ is isomorphic to Lafon's Enriques surface over $\mathbb{Q}(t)$. Let $C$ be a smooth projective curve over an algebraically closed field $K$ of characteristic zero and let $f \colon C \to \mathbb{P}^1_K$ be a finite morphism. Let $X_f$ be a smooth projective variety over $K$ birational to $X \times_{\mathbb{P}^1_{\mathbb{Q}}} C$. Then, the following properties hold.
\begin{enumerate}[nosep]
\item If $K = \mathbb{C}$, the natural homomorphism of topological fundamental groups $\pi_1(X_f) \to \pi_1(C)$ is an isomorphism.
\item The natural homomorphism of \'etale fundamental groups $\pi_1^{\et}(X_f) \to \pi_1^{\et}(C)$ is an isomorphism.
\end{enumerate}
\end{theorem}
\begin{proof}
To prove $(i)$, assume $K = \mathbb{C}$. Let $\alpha \in \mathbb{C}$ such that $\alpha^2 = \alpha + 1$ and let $c \in C$ be a preimage of $-\alpha$ under $C \to \mathbb{P}^1_\mathbb{C}$.
We consider the analytification $X_f^{\an}\to C^{\an}$ of $X_f\to C$.
Let $\Delta \subseteq C^{\an}$ be an open disk containing $c$ and let $c' \in \Delta$ be a point over which the fibre of $X_f \to C$ is smooth. Let $X_{f,\Delta}$ be the inverse image of $\Delta \subset C^{\an}$ along $X_f^{\an}\to C^{\an}$, and consider the morphism $X_{f,\Delta} \to \Delta$. Applying Lemma \ref{lemma:nori} to the fibre over $c'$, we see that the natural map $\pi_1(X_{f, c'}) \to \pi_1(X_{f, \Delta})$ is surjective. As $\pi_1(X_{f, c'}) \cong \mathbb{Z}/2\mathbb{Z}$, the map $\pi_1(X_{f, c'}) \to \pi_1(X_{f, \Delta})$ is hence either the zero map or an isomorphism. If it were an isomorphism, the K3 double cover of $X_{f, c'}$ would extend to a topological double cover $Y_{f, \Delta}$ of $X_{f, \Delta}$. Equipping $Y_{f, \Delta}$ with the unique complex-analytic structure making $Y_{f, \Delta} \to X_{f, \Delta}$ a local biholomorphism, we would see that the general fibres of $Y_{f, \Delta} \to \Delta$ are algebraic K3 surfaces.  The map $Y_{f, \Delta} \to X_{f, \Delta}$ would be an unramified double cover over an open neighbourhood of $c \in \Delta$. This would contradict Lemma \ref{lemma:k3cover_doesnt_extend}.
Thus, the map $\pi_1(X_{f, c'}) \to \pi_1(X_{f, \Delta})$ is the zero map, which implies that $\pi_1(X_{f, c'}) \to \pi_1(X_f)$ is the zero map as well. Again applying Lemma \ref{lemma:nori}, this time to the morphism $X_f \to C$ and its fibre $X_{f,c'}$, we conclude that $\pi_1(X_f) \to \pi_1(C)$ is an isomorphism, as desired. 

To prove $(ii)$, by the invariance of the \'etale fundamental group of a proper scheme over $K$ along extensions of algebraically closed fields \cite[Tag~0A49]{stacks-project} (or \cite[Corollaire~X.1.8]{SGA1}), we may assume that $K = \mathbb{C}$. Then the claim follows from $(i)$ and the fact that the \'etale fundamental group of $X_f$ is the pro-finite completion of $\pi_1$ (see \cite[Corollaire~XII.5.2]{SGA1}).
\end{proof}

\section{Extending a K3 double cover}\label{section:K3}

Recall that Lafon's Enriques surface $S$ over $\mathbb{Q}(t)$ is the minimal smooth projective model of the affine surface $U$ in $\mathbb{A}^4_{\mathbb{Q}(t)}$ defined by the equations 
\begin{equation*}
x^2 - tu^2 + t = ( t^2u^2 -t) y^2 \neq 0 \quad \text{and} \quad x^2 - 2tu^2 + t^{-1} = t( t^2u^2 - t) z^2 \neq 0.
\end{equation*}
Let $S'$ be the minimal smooth projective model of the affine surface defined by the single equation $w^2 =  t^2u^2 -t$ in $\mathbb{A}^1_U$, where $w$ is the new coordinate on $\mathbb{A}^1_U$. 
Note that $S'$ is a K3 surface and that the natural projection map $S'\to S$ is a finite \'etale cover of degree two. We will refer to $S'$ as \emph{Lafon's K3 surface (over $\mathbb{Q}(t)$}).

In this section, we will show that $S'\to S$ extends to a finite \'etale cover over the inf-multiple fibre of any regular proper model of $S$ over $\mathbb{P}^1_{\mathbb{Q}}$ at $t=0$.
By the already established properties of Lafon's Enriques surface (Theorem \ref{thm:main_insight}), this implies that the fibre above $t=0$ of any integral regular proper model of the K3 surface $S'$ over $\mathbb{P}^1_{\mathbb{Q}}$ is inf-multiple and non-divisible of inf-multiplicity two. (The existence of such K3 degenerations was not known previously, as we already mentioned in the introduction.) To put this in perspective, recall that no K3 double cover of $S$ extends to a finite \'etale cover over the singular fibre at $t=-\alpha$ with $\alpha$ the golden ratio (and this is what the proof of non-isotriviality given in Theorem~\ref{thm:lafon_is_non_isotrivial} is based on).
Thus, it may seem a bit surprising that there is such a finite \'etale extension at $t=0$ for $S'\to S$.
 
We start with verifying that for any finite extension $K$ of $\mathbb{C}\laurent{t}$, the $K$-rational points of $S$ lift to $S'$.
 
\begin{lemma}\label{lemma:K3_calculation}
For every integer $n\geq 1$ and every $u,x,y,z\in \mathbb{C}\laurent{t^{1/n}}$ such that
\begin{equation}\label{equation:lafon} \tag{$**$}
x^2 - tu^2 + t = ( t^2u^2 -t) y^2 \neq 0 \quad \text{and} \quad x^2 - 2tu^2 + t^{-1} = t( t^2u^2 - t) z^2 \neq 0,
\end{equation}
the element $t^2u^2-t$ is a square in $\mathbb{C}\laurent{t^{1/n}}$. 
\end{lemma}
\begin{proof}
Assume that there exists an integer $n\geq 1$ and elements $u,x,y,z$ of $\mathbb{C}\laurent{t^{1/n}}$ satisfying the two equations (\ref{equation:lafon}) with $t^2u^2-t$ not a square in $\mathbb{C}\laurent{t^{1/n}}$. Then, the first equation implies that $x^2-tu^2 +t$ is not a square (in $\mathbb{C}\laurent{t^{1/n}}$), and similarly the second equation implies that $x^2 t - 2t^2 u^2 + 1$ is not a square either. We now do a case-by-case analysis. Let $\nu$ be the $t$-adic valuation on $\mathbb{C}\laurent{t^{1/n}}$. 
\begin{itemize}
\item Assume $\nu(u) < -\frac{1}{2}$. Then, $\nu(t^2u^2) < \nu(t)$, so that $\nu(t^2 u^2 - t) = \nu(t^2 u^2)$ is divisible by two in $\frac{1}{n} \mathbb{Z}$, contradicting the fact that $t^2 u^2 - t$ is not a square.
\item Assume $\nu(u) = -\frac{1}{2}$ and $\nu(x) >0$. Then $-tu^2 $ is the unique lowest valuation term in $x^2-tu^2 +t$. Since $\nu(tu^2) =0$, this implies that $x^2-tu^2 +t$ is a square. Contradiction. 
\item Assume $\nu(u) = -\frac{1}{2}$ and $\nu(x) =0$. Then $1$ is the unique lowest valuation term in $x^2t - 2t^2 u^2 +1 $, which implies the latter is a square. Contradiction.
\item Assume $\nu(u) = -\frac{1}{2}$ and $\nu(x) <0$. Then $x^2$ is the unique lowest valuation term in $x^2-tu^2+t$, so that the latter is a square. Contradiction.
\item Assume $-\frac{1}{2} < \nu(u)<0$ and $\nu(x^2) < \nu(tu^2)$. Since $\nu(tu^2) <1$, it follows that $x^2$ is the unique lowest valuation term in $x^2 -tu^2 + t$. This implies the latter is a square. Contradiction.
\item Assume $-\frac{1}{2} < \nu(u)<0$ and $\nu(x^2) \geq \nu(tu^2)$. Since $\nu(tu^2)>0$, it follows that $1$ is the unique lowest valuation term in $x^2t - 2t^2u^2 +1$. This implies the latter is a square. Contradiction.
\item Assume $\nu(u)\geq 0$ and $\nu(x^2 t)\leq 0$. Then $x^2$ is the unique lowest valuation term in $x^2 - tu^2 + t$. This implies the latter is a square. Contradiction. 
\item Assume $\nu(u) \geq 0$ and $\nu(x^2 t) >0$. Then $1$ is the unique lowest valuation term in $x^2t - 2t^2 u^2 +1$. This implies the latter is a square. Contradiction.
\end{itemize} 
This concludes the proof.
\end{proof} 

Now that we have shown that local points lift, we can deduce that $S'\to S$ extends to a finite \'etale covering using the following lemma. 
 
\begin{lemma} \label{lemma:etale_follows_from_lifiting}
Let $R = \mathbb{C} \powerseries{t}$ and $K = \mathbb{C}\laurent{t}$, and let $\mathcal{X} \to \Spec R$ be a proper flat morphism with geometrically connected fibres, where $\mathcal{X}$ is an integral regular scheme. Let $\mathcal{Y} \to \mathcal{X}$ be a finite surjective morphism with $\mathcal{Y}$ an integral normal scheme. Assume that the fibres of $\mathcal{Y} \to \Spec R$ are geometrically connected. Assume $\mathcal{Y}_K \to \mathcal{X}_K$ is finite \'etale Galois. Then the following are equivalent.
\begin{enumerate}
\item For every finite extension $L/K$, the map $\mathcal{Y}(L) \to \mathcal{X}(L)$ is surjective.
\item There is a dense open $U \subseteq \mathcal{X}_K$ such that, for every finite extension $L/K$, the image of $\mathcal{Y}(L)\to \mathcal{X}(L)$ contains $U(L)$. 
\item The morphism $\mathcal{Y}\to \mathcal{X}$ is \'etale.
\end{enumerate}
\end{lemma} 
\begin{proof}  
Note that $(i) \implies (ii)$ trivially and that $(iii) \implies (i)$ follows from Lemma \ref{lemma:nonliftability_extendability}. Thus, it remains to prove $(ii) \implies (iii)$.

Let $\mathcal{X}_0$ be the special fibre of $\mathcal{X}\to \Spec R$ and write $\mathcal{X}_0 = \sum m_i X_i$, where the $m_i$ are positive integers and the $X_i$ are the irreducible components of $\mathcal{X}_0$. Similarly, write $\mathcal{Y}_0 = \sum n_{i,j} Y_{i,j}$, where the $n_{i,j}$ are positive integers and the $Y_{i,j}$ are the irreducible components of $\mathcal{Y}_0$ indexed such that $Y_{i,j}$ lies over $X_i$. To prove that $\mathcal{Y}\to \mathcal{X}$ is \'etale, it suffices to show that $m_i = n_{i,j}$ for all $i$ and $j$.

For each $i$, let $X_i^o$ be a dense open subset of $X_i$ such that $X_i^o$ meets no other irreducible components of $\mathcal{X}_0$, the subscheme $(X_i^o)_{\red}$ is regular, the morphism $\mathcal{Y}_{0,\red} \to \mathcal{X}_{0,\red}$ is \'etale over $(X_i^o)_{\red}$, and $\mathcal{Y}$ is regular above $X_i^o$. Let $x$ be any point in $X_i^o$ and choose an $R$-morphism $\Spec R[t^{1/m_i}] \to \mathcal{X}$ through $x$ such that the generic point of $\Spec R[t^{1/m_i}]$ maps to a point $\eta_x \in U$ (note that the existence of such a point follows from \cite[\S 9.1, Corollary~9]{BLR}). By assumption, there is an $\eta_y \in \mathcal{Y}(K[t^{1/m_i}])$ lying over $\eta_x$. Thus, as $\mathcal{Y} \to \Spec R$ is proper, we obtain a lift $\Spec R[t^{1/m_i}] \to \mathcal{Y}$ through $\eta_y$. Let $y \in \mathcal{Y}_0$ denote the image of the closed point of $\Spec R[t^{1/m_i}]$. As the morphism $\mathcal{Y}_{0,\red} \to \mathcal{X}_{0,\red}$ is \'etale over $x_{\red}$ and $x_{\red}$ is a regular point of $\mathcal{X}_{0,\red}$, the point $y_{\red}$ is a regular point of $\mathcal{Y}_{0,\red}$. In particular, it lies on a unique irreducible component $Y_{i,j_0}$. Moreover, as we assumed $\mathcal{Y}$ to be regular over $X_i^o$ and regularity is an open condition, the scheme $\mathcal{Y}$ is regular in a neighbourhood of $y_{\red}$. Let $B \subseteq \mathcal{Y}$ denote the image of $\Spec R[t^{1/m_i}]$ in $\mathcal{Y}$. Then we have the following equality of intersection numbers:
\[ m_i = \deg(B / \mathbb{C}\powerseries{t}) = (\mathcal{Y}_0 \cdot B) = n_{i,j_0} (Y_{i,j_0} \cdot B)\rlap{.} \] 
Thus, $n_{i,j_0}$ divides $m_i$, and therefore is equal to it. As $\mathcal{Y}_K \to \mathcal{X}_K$ was assumed to be Galois, we see that $n_{i,j} = m_i$ for every $j$. As $x$ was chosen to lie on an arbitrary $X_i^o$, we see that $n_{i,j} = m_i$ for every $i$ and $j$, finishing the proof.
\end{proof}

\begin{theorem} 
Let $f\colon X\to \mathbb{P}^1_{\mathbb{Q}}$ be a morphism with $X$ a smooth projective threefold over $\mathbb{Q}$ such that $X_{\mathbb{Q}(t)}$ is isomorphic to Lafon's Enriques surface. Let $Y\to X$ be the normalization of $X$ in the function field of $S'$ (so that $Y_{\mathbb{Q}(t)}$ is Lafon's K3 surface).   Then the following statements hold. 
\begin{enumerate}
\item There is a dense open $U\subset \mathbb{P}^1_{\mathbb{Q}}$ containing $0$ such that $Y \times_{\mathbb{P}^1_{\mathbb{Q}}} U \to X \times_{\mathbb{P}^1_{\mathbb{Q}}} U$ is finite \'etale over $U$.  
\item Let $\widetilde{Y}\to Y$ be a  resolution of singularities which is an isomorphism over the regular locus of $Y$. Then, the morphism $\widetilde{Y} \to \mathbb{P}^1_{\mathbb{Q}}$ has no divisible fibres and its orbifold base is $(\mathbb{P}^1_{\mathbb{Q}}, \frac{1}{2} [0])$. (That is, the fibre over $t=0$ is the unique inf-multiple fibre of $\widetilde{Y}\to \mathbb{P}^1_{\mathbb{Q}}$, it is not divisible, and it has inf-multiplicity two.) 
\item The family of K3 surfaces $Y\to \mathbb{P}^1_{\mathbb{Q}}$ is non-isotrivial.
\item The threefold $\widetilde{Y}_{\mathbb{C}}$ is simply connected. 
\end{enumerate} 
\end{theorem}
\begin{proof}
Consider the morphisms $\Spec \mathbb{C}\powerseries{t} \to \Spec \mathbb{Q}[t]_{(t)}\to \mathbb{P}^1_{\mathbb{Q}}$. Define $\mathcal{Y} = Y\times_{\mathbb{P}^1_{\mathbb{Q}}} \mathbb{C}\powerseries{t}$ and $\mathcal{X}=X\times_{\mathbb{P}^1_{\mathbb{Q}}} \mathbb{C}\powerseries{t}$. By standard descent arguments, to prove $(i)$, it suffices to show that $\mathcal{Y}\to \mathcal{X}$ is \'etale. To prove this, by Lemma \ref{lemma:etale_follows_from_lifiting}, it suffices to show that there is a dense open $U\subset \mathcal{X}_K$  such that, for every finite extension $L/\mathbb{C}\laurent{t}$, the image of $\mathcal{Y}(L)\to \mathcal{X}(L)$ contains $U(L)$. The latter is verified in Lemma \ref{lemma:K3_calculation}. This concludes the proof of $(i)$. 
 
To prove $(ii)$, we use that the points found in the proof of Theorem \ref{thm:main_insight} lift to the K3 double cover.
Namely, 
\[ x=0, u=0, y=\sqrt{-1}, z= \frac{\sqrt{-1}}{t\sqrt{t}}, w= \sqrt{-t} \]
defines a $\mathbb{Q}(\sqrt{-1})(\sqrt{t})$-point of $S'$ and, similarly, 
\[ x = 1, u = 1, y = \frac{1}{t \sqrt{1-t^{-1}}}, z = \frac{1}{t}\sqrt{\frac{-2+t^{-1}+t^{-2}}{1-t^{-1}}}, w= \sqrt{t^2-t}\]
defines a local point of $S'_{\mathbb{C}(t)}$ at $t=\infty$. It follows from Lemma \ref{lemma:local_points}   that $\widetilde{Y}_{t_0}$ has a reduced component for every $t_0 \in \mathbb{P}^1(\mathbb{C}) \setminus \{0\}$. This shows that $\widetilde{Y}\to \mathbb{P}^1_{\mathbb{Q}}$ has at most one inf-multiple fibre. Furthermore,  by $(i)$, the fibre $\widetilde{Y}_0 = Y_0$ is a finite \'etale cover of the inf-multiple non-divisible fibre $X_0$ with inf-multiplicity two. Therefore, $\widetilde{Y}_0$ is an inf-multiple non-divisible fibre with inf-multiplicity two as well.   (Alternatively,  note that the quadratic point and cubic point on $S_{\mathbb{C}\laurent{t}}$ constructed in the proof of Theorem \ref{thm:main_insight} lift to $S'_{\mathbb{C}\laurent{t}}$.)
 
To prove $(iii)$, suppose that $S'$ were isotrivial. Then $S'_{\overline{\mathbb{Q}(t)}}$ can be defined over $\overline{\mathbb{Q}}$. In particular, every Enriques quotient (such as $S_{\overline{\mathbb{Q}(t)}}$) can be defined over $\overline{\mathbb{Q}}$ by Lemma \ref{lemma:K3_quotients} below. 
This however contradicts the non-isotriviality of $S$ (Theorem~\ref{thm:lafon_is_non_isotrivial}). We conclude that $S'$ is non-isotrivial.

Since $\widetilde{Y}_{\mathbb{C}}\to \mathbb{P}^1_{\mathbb{C}}$ is a family of K3 surfaces with no divisible fibres, it follows immediately from Lemma~\ref{lemma:nori} that $\widetilde{Y}_{\mathbb{C}}$ is simply connected.   This establishes $(iv)$ and concludes the proof.  
\end{proof}

\begin{lemma} \label{lemma:K3_quotients}
Let $L/k$ be an extension of algebraically closed fields of characteristic zero. Let $V$ be a K3 surface over $k$. Then, every Enriques quotient of $V_L$ can be defined over $k$.
\end{lemma}
\begin{proof}
We may assume $k$ is countable and $L = \mathbb{C}$. Now, let $V_L\to W$ be an Enriques quotient of~$V_L$. 
Note that the set of isomorphism classes in the $\mathrm{Aut}_k(L)$-orbit of $W$ is finite, as a K3 surface over an algebraically closed field of characteristic zero has only finitely many Enriques quotients up to isomorphism (see \cite[Corollary~0.4]{Ohashi}). Thus, as $\sigma$ runs over the field automorphisms of $L$ trivial on $k$, the set of isomorphism classes among the conjugates $W^{\sigma}$ of $W$ is finite. In particular, it follows that $W$ can be defined over $k$ (see \cite[Criterion~1]{GonzalezDiez}), as required.   
\end{proof}

Note that Theorem \ref{thm1:K3_version} is a consequence of the following result, where we take $C=\mathbb{P}^1_{\mathbb{Q}}$ and $f\colon \mathbb{P}^1_{\mathbb{Q}} \to \mathbb{P}^1_{\mathbb{Q}}$ any degree $d$ morphism \'etale over $0$.

\begin{theorem}\label{thm2:K3_version}
Let $Y \to \mathbb{P}^1_{\mathbb{Q}}$ be a morphism with $Y$ a smooth projective threefold such that $Y_{\mathbb{Q}(t)}$ is isomorphic to Lafon's K3 surface. Let $C$ be a smooth projective curve over a field $K$ of characteristic zero. Let $f \colon C \to \mathbb{P}^1_K$ be a finite morphism which is \'etale over $0$. Let $Y_f$ be a desingularization of $Y \times_{\mathbb{P}^1_{\mathbb{Q}}} C$. Then, the following statements hold.
\begin{enumerate}
\item The threefold $Y_f$ is a smooth projective threefold, the generic fibre of $Y_f\to C$ is a K3 surface, the orbifold base of the morphism $Y_f\to C$ is $(C, \frac{1}{2} [f^{-1}(0)])$, and $Y_f \to C$ has no divisible fibres.
\item The threefold $Y_f$ is weakly special if and only if $C$ is of genus at most one.
\item The threefold $Y_f$ is special if and only if $C$ is of genus zero and $\deg f \leq 4$.
\end{enumerate}
\end{theorem} 
\begin{proof}
This is similar to the proof of Theorem \ref{thm2}.
\end{proof}

\bibliography{lafon}{}
\bibliographystyle{alpha}
\end{document}